\documentclass[12pt,onecolumn]{amsart}

\usepackage[margin=1in]{geometry}
\usepackage{pdfsync}

 \usepackage[plainpages=false]{hyperref}
 \usepackage{amsfonts,amsmath,amssymb,amsthm,accents,mathtools}
 \usepackage{latexsym,lscape,rawfonts,mathrsfs,wasysym}
 \usepackage{thmtools}
 
\usepackage[dvips]{color}
\usepackage{multicol}
\usepackage[dvipsnames]{xcolor}
\usepackage{tcolorbox}

 \usepackage[arrow,matrix, all, cmtip]{xy}
 \usepackage{psfrag}
 \usepackage{graphicx} 
 \usepackage{pstool}
 \usepackage{epstopdf}

 \usepackage{array,tabularx}

 \usepackage{appendix}
 
\usepackage{txfonts}

\usepackage{enumerate}
\usepackage{tikz}
\usetikzlibrary{arrows.meta}
\usepackage{cite}
\usepackage{url}
\usepackage{accents}
\usepackage{verbatim}

\newcommand{\im}{\operatorname{Im}}
\newcommand{\Tran}{\operatorname{Tran}}
\newcommand{\Span}{\operatorname{Span}}

\newcommand{\Sec}{\operatorname{Sec}}
\newcommand{\Prin}{\operatorname{Prin}}

\DeclareMathOperator{\supp}{supp}

\DeclareMathOperator{\Vol}{Vol}
\DeclareMathOperator{\diam}{diam}
\DeclareMathOperator{\Ric}{Ric}
\DeclareMathOperator{\inj}{inj}

\DeclareMathOperator{\id}{id}

\DeclareMathOperator{\Aut}{Aut}

\DeclareMathOperator{\Deck}{Deck}
\DeclareMathOperator{\Aff}{Aff}
\DeclareMathOperator{\rank}{rank}
\DeclareMathOperator{\Ker}{Ker}

\makeatletter
\def\ExtendSymbol#1#2#3#4#5{\ext@arrow 0099{\arrowfill@#1#2#3}{#4}{#5}}

\makeatother

\makeatletter
\def\ExtendSymbol#1#2#3#4#5{\ext@arrow 0099{\arrowfill@#1#2#3}{#4}{#5}}

\makeatother

\definecolor{orange}{rgb}{1,0.5,0}
\definecolor{brown}{rgb}{0.48,0.33,0.19}
\definecolor{magenta}{rgb}{1,0,1}
\definecolor{miao}{cmyk}{0.5,0,0.2,0.2}
\definecolor{qiao}{gray}{0.96}

\newtheorem{prop}{Proposition}[section]

\newtheorem{proposition}[prop]{Proposition}

\newtheorem{theorem}[prop]{Theorem}

\newtheorem{lemma}[prop]{Lemma}

\newtheorem{corollary}[prop]{Corollary}

\newtheorem{remark}[prop]{Remark}
\newtheorem{example}[prop]{Example}

\newtheorem{definition}[prop]{Definition}

\newtheorem{conjecture}[prop]{Conjecture}

\newtheorem{property}[prop]{Property}

\newtheorem*{proof of theorem 1.2}{Proof of Theorem 1.2}
\newtheorem*{proof of theorem A.1}{Proof of Theorem A.1}

\numberwithin{equation}{section}

\title[Structure of Torus Fibration Under the First Betti Number Restriction]{Structure of Torus Fibration Under the First Betti Number Restriction}
\author{Xin Peng}
\author{Bing Wang}
\author{Zhenjian Wang}

\address{Xin Peng, School of Mathematical Sciences,
University of Science and Technology of China;
Hefei National Laboratory, Hefei 230088, China
}
\email{px2333@mail.ustc.edu.cn}

\address{Bing Wang,  Institute of Geometry and Physics, and School of Mathematical Sciences,
University of Science and Technology of China, Hefei 230026, China;
Hefei National Laboratory, Hefei 230088, China}

\email{topspin@ustc.edu.cn}
\address{Zhenjian Wang, Hefei National Laboratory, Hefei 230088, China}
\email{wzhj@ustc.edu.cn}

\date{}

\begin{document}
	
	\begin{abstract}
	We study torus bundles with affine structure groups. First, we establish a rigidity result under constraints on the first Betti numbers: If $ \text{b}_{1}(M)-\text{b}_{1}(N)=\dim M-\dim N $ holds for a torus bundle $M$ with an affine structure group over a closed manifold $N$, then $M$ can be classified. 
     Second, we obtain some necessary and sufficient conditions for the topological splitting of principal torus bundles. 
       These results improve the understanding of the geometry of collapsing sequences under the first Betti number constraints, thereby extending the prior work by Huang-Wang~\cite{HuangWang20.1}.
    \end{abstract}
    
    \maketitle
    \tableofcontents

    \section{Introduction}
      \unskip
      ~~~Let $\mathcal{M}(m,D):=\{(M^{m},g)|\Ric(g)\geq-(m-1)g,~\diam(M,g)\le D\}$ denote a class of $m$-dimensional compact Riemannian manifolds with a uniform lower bound on the Ricci curvature and a uniform upper bound on diameter. Similarly, let 
      $$\mathcal{M}(k,D,\upsilon):=\{(N^k,h)||\Sec(h)|\le 1 ,  \diam(N,h)\le D ,  \Vol(N,h) \ge \upsilon>0\}$$
      denote a class of $k$-dimensional compact Riemannian manifolds with bounded sectional curvature, diameter, and volume. By Gromov's precompactness theorem (see\cite{GLP}), any sequence $\{(M_{i},g_{i})\}_{i\in \mathbb{N}^{+}}\subseteq \mathcal{M}(m,D)$ admits a subsequence converging to a metric space $(X,d)$ in the Gromov-Hausdorff sense. We are concerned with the topological and geometric properties of such convergent sequences $(M_{i},g_{i})$ and their limit spaces $(X,d)$; for related topological results, we refer to \cite{CGI},\cite{CGII},\cite{CFG92},\cite{Gromov78b},\cite{KW11},\cite{NaberZhang},\cite{Ruh}.

      Consider the case where a sequence $\{(M^{m}_{i},g_{i})\}_{i\in \mathbb{N}^{+}}\subseteq \mathcal{M}(m,D)$ converges to a compact manifold $(N^{k},h)\in \mathcal{M}(k,D,\upsilon)$ in the Gromov-Hausdorff sense. In this setting, S.S. Huang and B. Wang~\cite[Theorem 1.1]{HuangWang20.1} show that the difference between the first Betti number of $M$ and $N$ is no larger than the difference in dimension between $M$ and $N$ (See also Rong's proof in~\cite[Theorem 0.10]{Rong22}).
      Moreover, if the equality holds, then $M$ is diffeomorphic to a torus bundle over $N$. The torus bundle appearing in~\cite[Theorem 1.1]{HuangWang20.1} is an instance of a Fukaya fibration; for other types, see \cite{Fukaya87ld},\cite{Fukaya88},\cite{HuangWang20.2},\cite{Huang20},\cite{HKRX18},\cite{NaberZhang}. Furthermore, Fukaya’s analysis~\cite[Theorem 0.1]{Fukaya89} shows the structure group of such a torus bundle is $T^{m-k} \rtimes \Aut (T^{m-k})$.
       
    In this paper, we further investigate the structure of such torus bundles in the equality case \begin{equation}
                    \text{b}_{1}(M)-\text{b}_{1}(N)=\dim M-\dim N\tag{*}~.\label{eq:difference} 
                \end{equation} 
     Specifically, we address the following questions:
    \begin{enumerate}[(Q1)]
        \item For a given base manifold $N$, can we classify all such total spaces $M$?
        \item Under what conditions is $M$ diffeomorphic to the product $N \times T^{m-k}$?
    \end{enumerate}
    
    Our first main result characterizes the structure of such torus bundles when the Betti number difference equals the dimension difference:
    \begin{theorem}[Structure Theorem under Difference of the First Betti Numbers]\label{the: Structure Theorem}
    		Let $N^{k}$ be a smooth closed manifold and $M^{m}$ a smooth torus bundle over $N$ satisfying condition~(\ref{eq:difference}) and the smallest structure group of $M$ is contained in the \textbf{affine group} $T^{m-k} \rtimes \Aut (T^{m-k})$.
    	    Then:
    	 	\begin{enumerate}[(1)]
    	 	\item $M$ is a principal $T^{m-k}$-bundle over $N$, equipped with a smooth $T^{m-k}$-action;
    	 	\item There exist finite normal covering spaces $\hat{M}$ of $M$ and $\hat{N}$ of $N$ of the same index with abelian deck transformation groups, such that $\hat{M}$ is diffeomorphic to the product $\hat{N} \times T^{m-k}$.  
            \end{enumerate}
    \end{theorem}

     Theorem \ref{the: Structure Theorem}(2) implies that $M$ is ``almost'' a product of $N$ and $T^{m-k}$. Moreover, Theorem \ref{the: Structure Theorem}(1) shows that $M$ is a principal bundle over $N$. Using the Euler class, we can further distinguish the topology of $M$. 

     We point out that there exists an example where $M$ is not diffeomorphic to the product $N \times T^{m-k}$. Let $M:=S^{3}\times S^{1}/ \left \langle \alpha' \right \rangle$ and $N:= S^{3}/\left \langle \alpha \right \rangle$, 
     where \begin{equation}
         \alpha':S^{3}\times S^{1}\to S^{3}\times S^{1},\alpha'(x,y,e^{i\theta})=(-x,-y,-e^{i\theta})
     \end{equation}
     and \begin{equation}
         \alpha:S^{3}\to S^{3}, \alpha(x,y)=(-x,-y)
     \end{equation}for any $(x,y)\in S^{3}\subseteq \mathbb{C}^{2},\theta\in [0,2\pi)$. Then, the map 
     \begin{equation}
     f:M\to N, f([(x,y,e^{i\theta})])=[(x,y)], \forall [(x,y,e^{i\theta})]\in M, [(x,y)]\in N
     \end{equation} gives the toric fibration over $N$ with structure group isomorphic to $\mathbb{Z}_{2}=\{1,-1\}<S^{1}$.
     Computing the fundamental groups, we obtain condition (\ref{eq:difference}) and $\pi_{1}(M)\cong \mathbb{Z}$, which is torsion-free, while $\pi_{1}(N)\cong \mathbb{Z}_{2}$ contains torsion. Therefore, $M$ cannot be diffeomorphic to the product $N \times S^{1}$.

    Our second main result gives a sufficient and necessary condition for $M$ to be a product:
    
    \begin{theorem}[Topological Splitting Theorem for Principal $ T^{m-k}$-Bundles] \label{the:Topological Splitting Theorem}
    		Given a smooth closed manifold $N^{k}$ and a smooth principal $T^{m-k}$-bundle $M^{m}$ over $N$ with smooth $T^{m-k}$-action.
    	If the homotopy exact sequence of the bundle can be decomposed into a short splitting exact subsequence
        \begin{equation}
         \xymatrix{
         0 \ar[r] & \pi_{1}(T^{m-k}) \ar[r]^-{l_{*}} & \pi_{1}(M) \ar[r]^{f_{*}} & \pi_{1}(N) \ar[r] & 0 
         }\label{se: short exact sequence (2)},
         \end{equation}
        then $M$ is diffeomorphic to the product $N\times T^{m-k}$.	Here, $l:T^{m-k}\to M$ is an embedding of the fiber and $f:M\to N$ is the bundle projection.
          
        In particular, if $N$ is simply connected, then $M$ 
        is a trivial bundle.
    \end{theorem}

     Analyzing the homotopy groups of $M$ and $N$, we observe that there exist normal covering spaces $\hat{M}$ and $\hat{N}$ of $M$ and $N$, respectively, they satisfy condition (\ref{se: short exact sequence (2)}). Consequently, $\hat{M}\cong \hat{N}\times T^{m-k}$.
    
    \begin{corollary}\label{cor 1.3} 
    Let $N^{k}$ be a smooth closed manifold and $M^{m}$ a smooth principal $T^{m-k}$-bundle over $N$ with a smooth $T^{m-k}$-action. If the sequence
        \begin{equation}
         \xymatrix{
         0 \ar[r] & H_{1}(T^{m-k}) \ar[r]^-{\bar{l}_{*}} & H_{1}(M) \ar[r]^{\bar{f}_{*}} & H_{1}(N) \ar[r] & 0 
         }\label{se: short exact sequence (3)}
         \end{equation}
          induced by the bundle is splitting, then $M$ is diffeomorphic to the product $N\times T^{m-k}$.
    \end{corollary}

       Note that the assumption in Corollary~\ref{cor 1.3} differs from that in Theorem~\ref{the:Topological Splitting Theorem}. Nevertheless, Corollary~\ref{cor 1.3} still provides a necessary and sufficient condition for the topological splitting of the principal $T^{m-k}$-bundle. We observe that the only obstacle to making $M$ a product of $N$ and $T^{m-k}$ ``lies'' in the torsion part of $H_{1}(N)$. In fact, it follows from the Universal Coefficient Theorem that the obstacle ``lies'' in the torsion part of $H^{2}(N;\mathbb{Z}^{m-k})$.\\     

        The proof of Theorem~\ref{the: Structure Theorem}(1) leverages a structural simplification afforded by the topological condition on Betti numbers. Specifically, when the base space $N$ is the circle $S^{1}$, this condition forces the vanishing of the affine component in the structure group of bundle $M$, reducing it to a torus. This reduction is pivotal, as it establishes that the fibration is a principal torus bundle. The principal nature of the bundle provides a rigid framework for the subsequent analysis.

        The core of the argument then proceeds through a homotopy-theoretic construction. The triviality of the boundary map in the associated homotopy exact sequence allows for the construction of finite covering spaces $\hat{M}$ and $\hat{N}$ of $M$ and $N$, respectively. On these covers, the problem becomes tractable via obstruction theory. Applying this theory shows that the covering space $\hat{M}$ is diffeomorphic to a topological product of $\hat{N}$ and $T^{m-k}$, from which the global structure of the original manifold $M$  can be fully deduced.\\

    This paper is organized as follows:

    In Section \ref{sec: Prel}, we introduce background material used in subsequent sections. In Section \ref{sec: Homo}, we analyze the homotopy exact sequence of the torus bundle under assumption (\ref{eq:difference}), construct the finite normal covering spaces $\hat{M}$ and $\hat{N}$ as in Theorem~\ref{the: Structure Theorem}, and verify that they satisfy conditions $(\ref{se: short exact sequence (2)})$ of Theorem~\ref{the:Topological Splitting Theorem}. In Section \ref{sec: Pri}, under the assumptions of Theorem~\ref{the: Structure Theorem}, we prove that the smallest structure group of the bundle reduces to $T^{m-k}$. Next, we prove Theorem~\ref{the:Topological Splitting Theorem} by computing the Euler class of the bundle. In Section \ref{sec: Comp}, we analyze the deck transformation groups of the finite normal covers and classify all such $M$ in Theorem~\ref{the: Structure Theorem}. In Section \ref{sec: Fur}, we discuss further results and conjectures related to the Betti numbers and the topological splitting theorem for general principal $G$-bundles.\\

    \textbf{Acknowledgements:}  Xin Peng thanks Yu Li for helpful discussions. Xin Peng and Bing Wang are supported by YSBR-001, NSFC-12431003 and a research fund from USTC and Hefei National Laboratory.  Zhenjian Wang is supported by NSFC-12301100.
    
    \section*{Notations and Conventions}
    Throughout this paper, we adhere to the following notations and conventions:
   \begin{enumerate}
        \item[$\bullet$] $f: M \to N$ denotes the torus bundle with an affine structure group. The fiber is the $(m-k)$-dimensional torus $T^{m-k}$, and $l: T^{m-k} \to M$ denotes a fiber embedding.
   	  \item[$\bullet$] $\hat{M}$ and $\hat{N}$ denote the finite 
        normal covering spaces of $M$ and $N$, respectively, constructed in Section \ref{sec: Homo}. 
        $\text{pr}_M: \hat{M} \to M$ and $\text{pr}_N: \hat{N} \to N$ denote the covering maps.     
        \item[$\bullet$] $\text{pr}_{\hat{N}}:\hat{M}\to \hat{N}$ denotes the torus fibration (see section~\ref{sec: Homo}), $\text{pr}_{T^{m-k}}:\mathbb{R}^{m-k}\to T^{m-k}$ denotes the standard covering map and $p_{T^{m-k}}:\hat{M}(\cong \hat{N}\times T^{m-k})\to T^{m-k}$ denotes the fiber component projection (see Section~\ref{sec: Comp}).
        \item[$\bullet$] For a covering map $p: \widetilde{X} \to X$, the deck transformation group is denoted by $\Deck(\widetilde{X}/X)$. The element $[0]$ denotes the identity of $T^{m-k}$.
        \item[$\bullet$] $h:\pi_{1}(M)\to H_{1}(M)$, $h':\pi_{1}(N)\to H_{1}(N)$ and $h'':\pi_{1}(T^{m-k})\to H_{1}(T^{m-k})$ denote the Hurewicz homomorphisms (see Section \ref{sec: Homo}).
        \item[$\bullet$] $\text{Pr}_M: \pi_1(\hat{M}) \to \pi_1(M)$ and $\text{Pr}_N: \pi_1(\hat{N}) \to \pi_1(N)$ denote the homomorphisms induced on fundamental groups by the covering maps.
        
        $\text{Pr}_{M}^{*}:\pi_{1}(M)\to \Deck(\hat{M}/M)$, $\text{Pr}_{N}^{*}:\pi_{1}(N)\to \Deck(\hat{N}/N)$ and $\text{Pr}_{\hat{N}\times\mathbb{R}^{m-k}/M}^{*}:\pi_{1}(M)\to \Deck(\hat{N}\times\mathbb{R}^{m-k}/M)$ denote the quotient homomorphisms. 
        
        $\text{Pr}_{\hat{M}/M}:\Deck(\hat{N}\times \mathbb{R}^{m-k}/\hat{M})\to \Deck(\hat{N}\times \mathbb{R}^{m-k}/M)$ and $(\text{Pr}_{T^{m-k}})_{*}:\Deck(\hat{N}\times \mathbb{R}^{m-k}/M)\to \Deck(\hat{M}/M)$ denote homomorphisms between deck transformation groups (see Section \ref{sec: Comp}).
\end{enumerate}
   
\section{Preliminaries}\label{sec: Prel}

   In this section, we review background material on fiber bundles, covering spaces, and obstruction theory, which will be used throughout the paper.
   
    \subsection{Homotopy theory}
    We begin by recalling several fundamental results in homotopy theory.
    \begin{theorem}[Covering Homotopy Theorem {\cite[Subsection 11.7]{Steenrod}}]\label{the:Covering Homotopy Theorem}	Let $B$ be a fiber bundle over a topological space $X^{'}$, and $X$ another topological space such that any open covering of $X$ can be reduced to a countable covering. Let $h:X\times[0,1]\to X^{'}$ be a homotopy from $h_{0}=h(\cdot,0)$ to $h_{1}=h(\cdot,1)$, and $H_{0}:X\to B$ a lift of $h_{0}$. Then there exists a lift $H:X\times[0,1]\to B$ of $h$ with $H(\cdot,0)=H_{0}$.

    The commutative diagram below illustrates this lifting property:
   \begin{equation}
       \xymatrix{
             X \ar[r]^{H_{0}} \ar[d]_{(\id_{X},0)}& B \ar[d]^{p}\\
      X\times[0,1] \ar@{-->}[ur]^{H} \ar[r]^-{h} &X'.
       }
   \end{equation}
\end{theorem}
    \begin{theorem}[See{\cite[Section 17]{Steenrod}}]\label{Thm2.2}
         Let $p:B\to X$ be a fiber bundle with fiber $F$. Then there is a long exact sequence of homotopy groups:
         \begin{equation}
         \xymatrix{
         \dots \ar[r]^{\partial_{i}} & \pi_{i}(F) \ar[r]^{l_{*}} & \pi_{i}(B) \ar[r]^{p_{*}} & \pi_{i}(X) \ar[r]^-{\partial_{i-1}} & \pi_{i-1}(F) \ar[r]^-{l_{*}} & \dots ~~~\text{for each $i\ge 1$}
         }.
         \end{equation}
    \end{theorem}
    The following result on the reduction of structure groups will be crucial in Section \ref{sec: Pri}: 
    
   \begin{theorem}[Reduction Theorem {\cite[Chapter 2, Section 7]{KN63}}]\label{Red Thm}Given a smooth fiber bundle $M$ over a connected base $N$ with a finite-dimensional structure Lie group $G$, let $\Gamma$ be a connection on $M$. Then the structure group $G$ can be reduced to the holonomy group of $\Gamma$.
   \end{theorem}
   An immediate consequence is:
   \begin{corollary}\label{Red Thm:cor}
     Every element of the minimal structure group of a smooth fiber bundle with a finite-dimensional structure Lie group can be realized in the structure group of the pullback bundle over a loop in $N$.
   \end{corollary}
   
    We also require the standard lifting criterion for covering spaces.

    \begin{proposition}[Lifting Criterion {\cite[Chapter 1 section 1.3]{Hatcher}}]\label{prop: lifting cri}
   Given a covering space $p:(\widetilde{X},\widetilde{x}_{0})\to(X,x_{0})$ and a map $F:(Y,y_{0})\to(X,x_{0})$ with $Y$ path-connected and locally path-connected. Then a lift $\widetilde{F}:(Y,y_{0})\to (\widetilde{X},\widetilde{x}_{0})$ of $F$ exists if and only if 
    $F_{*}(\pi_{1}(Y,y_{0}))\subseteq p_{*}(\pi_{1}(\widetilde{X},\widetilde{x}_{0}))$.      
    \end{proposition}
   Next, we briefly review the theory of normal covering spaces, this property will be frequently used in Section \ref{sec: Comp}.
     \begin{definition}[See {\cite[Chapter 1 section 1.3]{Hatcher}}]A covering space $p:\widetilde{X}\to X$ is called \textbf{normal}, if for each $x\in X$ and any $\widetilde{x},\widetilde{x}'\in p^{-1}(x)$, there exists an $\alpha\in \Deck(\widetilde{X}/X)$ taking $\widetilde{x}$ to $\widetilde{x}'$. 
     \end{definition}  
     
    \begin{property}[See {\cite[Chapter 1 section 1.3]{Hatcher}}]\label{property: normal}
        \begin{enumerate}[(1)]
            \item Given a covering space $p:\widetilde{X}\to X$ over a path-connected, locally path-connected space $X$. If $\widetilde{X}$ is also path-connected,  then:
            \begin{enumerate}
                \item This covering space is normal if and only if $p_{*}(\pi_{1}(\widetilde{X}))$ is a normal subgroup of $\pi_{1}(X)$;
                \item If $\widetilde{X}$ is normal, then $\Deck(\widetilde{X}/X)\cong \frac{\pi_{1}(X)}{p_{*}(\pi_{1}(\widetilde{X}))}$.
            \end{enumerate}
            \item  Suppose a group $G$ acts freely and properly discontinuously on a path-connected and locally path-connected space $Y$. Then:
            \begin{enumerate}
                \item The quotient map $p:Y\to Y/G$, $p(y)=Gy$ is a normal covering;
                \item $G\cong\Deck(Y/(Y/G))\cong\frac{\pi_{1}(Y/G)}{p_{*}(\pi_{1}(Y))}$.
            \end{enumerate}
        \end{enumerate}
    \end{property}
\subsection{Obstruction Theory and Euler Class}
   
    In this subsection, we summarize the key ideas of obstruction theory needed to define the Euler class.
   
    Let $X$ be a CW complex and $Y$ a topological space. Suppose a map $f^{l}:X^{l}\to Y$ is defined on the $l$-skeleton. To extend $f^{l}$ to the $(l+1)$-skeleton, we consider each $(l+1)$-cell $e^{l+1}_{\alpha}:D^{l+1}_{\alpha}\to X$. The obstruction to extending over this cell is measured by the composition $f^{l}\circ e^{l+1}_{\alpha}|_{\partial D_{\alpha}^{l+1}}:\partial D^{l+1} \to Y$, which defines an element of $\pi_{l}(Y)$. 

     Let $X^{l+1}=\{e^{l+1}_{\alpha}\}_{\alpha\in \Lambda_{l+1}}$. The obstruction cochain $\mathcal{O}(f^{l})\in C^{l+1}(X;\pi_{l}(Y))$ is defined by
     \begin{equation}
         \mathcal{O}(f^{l})(e^{l+1}_{\alpha}):=[f^{l}\circ e^{l+1}_{\alpha}|_{\partial D^{l+1}_{\alpha}}]\in \pi_{l}(Y).
     \end{equation}

       \begin{lemma}[See \cite{Steenrod}]\label{lem2.8}The obstruction cochain $\mathcal{O}(f^{l})$ is a cocycle. Moreover, if $g^{l}:X^{l}\to Y$ is another map agreeing with $f^{l}$ on the $(l-1)$-skeleton, then $\mathcal{O}(f^{l})-\mathcal{O}(g^{l})$ is a coboundary. Conversely, for any coboundary $dc$, there exists a map $g^{l}:X^{l}\to Y$ such that $\mathcal{O}(f^{l})-\mathcal{O}(g^{l})=dc$.
       \end{lemma}

       Thus, the obstruction cochain defines a cohomology class $[\mathcal{O}(f^{l})]\in H^{l+1}(X;\pi_{l}(Y))$, which vanishes if and only if $f^{l}$ can be extended on the $(l+1)$-skeleton. 

\begin{definition}
      The class $[\mathcal{O}(f^{l})]$ is called the $l$-skeleton obstruction class. 
\end{definition}

      We now apply this framework to the problem of constructing sections of fiber bundles. Let $p:P\to X$ be a fiber bundle over a CW complex $X$ with fiber $F$. Suppose $\sigma^{l}:X^{l}\to P$ is a section defined on the $l$-skeleton. For an $(l+1)$-cell $e^{l+1}_{\alpha}:D^{l+1}_{\alpha}\to X$, the composition $\sigma ^{l}\circ e^{l+1}_{\alpha}|_{\partial D^{l+1}_{\alpha}}$ gives a map into $P$ whose projection to $X$ is $e^{l+1}_{\alpha}|_{\partial D^{l+1}_{\alpha}}$.
      Using a local trivialization of the bundle over $D^{l+1}_{\alpha}$, this map can be projected to the fiber
      $F$, defining an element of $\pi_{l}(F)$, since the attaching map is null-homotopic in $X$. Thus, the obstruction to extending the section lies in $H^{l+1}(X;\pi_{l}(F))$.
      
      We now define the Euler class via obstruction theory. Consider a principal $S^{1}$-bundle $P$ over a CW complex $X$. Constructing a section skeleton-wise, the primary obstruction arises when extending from the 1-skeleton to the 2-skeleton.

       Suppose we have constructed the $l$-skeleton section $\sigma^{l}:X^{l}\to P$. We now want to extend it to an $(l+1)$-skeleton section using the Covering Homotopy Theorem~\ref{the:Covering Homotopy Theorem}:
        For each $(l+1)$-cell $e^{l+1}_{\alpha}:D^{l+1}_{\alpha}\to X$, we divide into two cases:
\begin{enumerate}[1.]
    \item When $l\ne1$, there is an extension $\sigma_{\alpha}^{l+1}:D^{l+1}_{\alpha}\to P$ of $(\sigma^{l}\circ e^{l+1}_{\alpha}|_{\partial D^{l+1}_{\alpha}})$;
    \item When $l=1$, there is an extension $\sigma_{\alpha}^{l+1}:D^{l+1}_{\alpha}\to P$ of $(\sigma^{l}\circ e^{l+1}_{\alpha}|_{\partial D^{l+1}_{\alpha}})$ if and only if $\im((\sigma^{l}\circ e^{l+1}_{\alpha}|_{\partial D^{l+1}_{\alpha}})_{*})=\{e\}$, where $(\sigma^{l}\circ e^{l+1}_{\alpha}|_{\partial D^{l+1}_{\alpha}})_{*}:\pi_{1}(\partial D^{l+1}_{\alpha})\to \pi_{1}(P)$ is the induced homomorphism on fundamental groups. 
\end{enumerate}

 \begin{definition} For a principal $S^{1}$-bundle $P\to X$, the 
       \textbf{Euler class} is the obstruction class $C_{1}(P)=[\mathcal{O}(\sigma^{1})]\in H^{2}(X;\mathbb{Z})$ to extending a section $\sigma^{1}:X^{1}\to P$ from the 1-skeleton.
 \end{definition}

     This definition generalizes to principal torus bundles:
 \begin{definition} For a principal $T^{m-k}$-bundle $P\to X$, the 
        \textbf{Euler class} is the obstruction class $C_{1}(P)=[\mathcal{O}(\sigma^{1})]\in H^{2}(X;\mathbb{Z}^{m-k})$ to extending a section $\sigma^{1}:X^{1}\to P$ from the 1-skeleton.
 \end{definition}
 
    The classification of principal bundles in the topological category is given by the following classical result.
    
\begin{theorem}[Classification of Principal Bundles {\cite[Section 19.3]{Steenrod}}]\label{Thm:Classification 1}
    Let $\mathfrak{B}\to X$ be an $(n+1)$-universal $G$-bundle for the group $G$, and $K$ an $n$-complex. Then the map
    \begin{equation}
           [g]\mapsto [g^{-1}(\mathfrak{B})]
    \end{equation}
    induces a bijection between the set of homotopy classes $[K,X]$ and the set of isomorphism classes of principal $G$-bundles over $K$.
\end{theorem}
       Let $X=K(\mathbb{Z}^{m-k},2)(\cong \prod_{i=1}^{m-k} K(\mathbb{Z},2)\cong \prod_{i=1}^{m-k}\mathbb{C}\mathbb{P}^{\infty})$ be an Eilenberg–MacLane space (see \cite[Section 4.2]{Hatcher}) and $K=N$. The universal $T^{m-k}$-bundle $\mathfrak{B}$ exists (indeed, $\mathfrak{B}\cong\prod_{i=1}^{m-k} S^{\infty}$), so Theorem \ref{Thm:Classification 1} applies. We identify the set of homotopy classes $[N,K(\mathbb{Z}^{m-k},2)]$ with a cohomology group via the following theorem.
\begin{theorem}[See {\cite[Section 4.3. Theorem 4.57]{Hatcher}}]
    There is a natural bijection
    \begin{equation}
        \begin{aligned}
            T:[X,K(G,n)]&\to H^{n}(X;G)\\
            [g]&\mapsto g^{*}(\alpha)
        \end{aligned}
    \end{equation} for any abelian group $G$ and certain distinguished class $\alpha\in H^{n}(K(G,n);G)$.
\end{theorem}
      When $G=\mathbb{Z}$ and $n=2$, the class $\alpha$ corresponds to the dual class of the exceptional divisor in $H_{2}(\mathbb{C}\mathbb{P}^{\infty};\mathbb{Z})$. The Euler class $C_{1}(g^{-1}(\mathfrak{B}))$ coincides with $g^{*}(\alpha)$ (see \cite[Chapter 4,5,9 and 14]{MS}).
      We therefore obtain:
     \begin{corollary}\label{Cor:2.14}
         The set of isomorphism classes of principal $T^{m-k}$-bundles over $N$ is in bijection with $H^{2}(N;\mathbb{Z}^{m-k})$. In particular, a principal $T^{m-k}$-bundle $M$ over $N$ is trivial if and only if its Euler class $C_{1}(M)$ vanishes.
     \end{corollary}

    \begin{remark}
        When $N$ is smooth, all principal $T^{m-k}$-bundles over $N$ in Corollary~\ref{Cor:2.14} can be taken in the smooth category, see \cite[Section 4.3 in Preliminaries]{Raghunathan} and Theorem \ref{Thm: A.1} in appendix.
    \end{remark}
   
	\section{Homotopic Exact Sequence of Torus Bundles}\label{sec: Homo}
        In this section, we analyze the homotopy exact sequences of torus bundles under the assumptions of Theorem~\ref{the: Structure Theorem}. We construct the covering spaces $\hat{M}$ and $\hat{N}$ and establish their key properties.
    
	\begin{theorem}\label{Thm3.1}
		Let $f:M\to N$ be a torus bundle with fiber $T^{m-k}$, where $M$ and $N$ are closed manifolds satisfying condition (\ref{eq:difference}). Then the following hold:
        \begin{enumerate}[(1)]
		\item The boundary maps $ \partial_{i}: \pi_{i+1}(N)\to \pi_{i}(T^{m-k})$ are trivial for all $i\ge1$. Consequently, the long homotopy exact sequence induced by $f$ yields isomorphisms $f_{*}:\pi_{i}(M)\cong \pi_{i}(N) $ for $ i\ge2 $, and a short exact sequence (\ref{se: short exact sequence (2)}) of fundamental groups;
		\item $l_{*}(\pi_{1}(T^{m-k}))\cap [\pi_{1}(M),\pi_{1}(M)]=\{e\} $;
		\item The image $l_{*}(\pi_{1}(T^{m-k}))$ lies in the center of $\pi_{1}(M)$.
		\end{enumerate}
	\end{theorem}

\begin{proof}
	Consider the commutative diagram induced by the Hurewicz homomorphisms:
    \begin{equation}
       \xymatrix{
       \pi_{2}(N) \ar[r]^-{\partial_{1}} & \pi_{1}(T^{m-k}) \ar[r]^-{l_{*}} \ar@{=}[d]^{h''} &\pi_{1}(M) \ar[r]^{f_{*}} \ar[d]^{h} &\pi_{1}(N) \ar[r] \ar[d]^{h'} &0\\
         & H_{1}(T^{m-k}) \ar[r]^-{\bar{l}_{*}} & H_{1}(M) \ar[r]^{\bar{f}_{*}} \ar@{=}[d] & H_{1}(N) \ar@{=}[d] \ar[r]&0\\
         & & \frac{\pi_{1}(M)}{[\pi_{1}(M),\pi_{1}(M)]} & \frac{\pi_{1}(N)}{[\pi_{1}(N),\pi_{1}(N)]}&.
         }\label{big diagram}
    \end{equation}
  We first show that $f_{*}^{-1}([\pi_{1}(N),\pi_{1}(N)])=\Ker(f_{*})\cdot[\pi_{1}(M),\pi_{1}(M)]$. 
   \begin{enumerate}[(a')]
       \item $(\Ker(f_{*})\cdot[\pi_{1}(M),\pi_{1}(M)]\subseteq f_{*}^{-1}([\pi_{1}(N),\pi_{1}(N)]))$: 
       \begin{equation}
       \begin{aligned}
             &f_{*}(\Ker(f_{*})\cdot[\pi_{1}(M),\pi_{1}(M)])=f_{*}(\Ker(f_{*}))\cdot f_{*}([\pi_{1}(M),\pi_{1}(M)])\\
             =&[f_{*}(\pi_{1}(M)),f_{*}(\pi_{1}(M))]=[\pi_{1}(N),\pi_{1}(N)];
       \end{aligned}    
       \end{equation}
       \item $(f_{*}^{-1}([\pi_{1}(N),\pi_{1}(N)])\subseteq \Ker(f_{*})\cdot[\pi_{1}(M),\pi_{1}(M)])$: For each $a\in(f_{*}^{-1}([\pi_{1}(N),\pi_{1}(N)])$, we write $f_{*}(a)=[\text{b}_{1},\text{b}_{1}']\dots[\text{b}_{l},\text{b}_{l}']$. By surjectivity of $f_{*}$, we can find some $a_{1},a_{1}',\dots, a_{l},a_{l}'$ such that $f_{*}(a_{i})=\text{b}_{i},f_{*}(a_{i}')=\text{b}_{i}'$.
       Then $f_{*}([a_{1},a_{1}']\dots[a_{l},a_{l}']\cdot a^{-1})=e$, hence exists a $t\in \Ker(f_{*})$ such that $a=t\cdot [a_{1},a_{1}']\dots[a_{l},a_{l}']$. Therefore, $(f_{*}^{-1}([\pi_{1}(N),\pi_{1}(N)])\subseteq \Ker(f_{*})\cdot[\pi_{1}(M),\pi_{1}(M)])$.
   \end{enumerate}
   \bigskip
    Using the surjectivity of $\bar{f}_{*}$ and $h$, we next compute the kernel of $\bar{f}_{*}$:
  \begin{equation}
	\begin{aligned}
		\Ker(\bar{f}_{*})
        &\cong\frac{\Ker(\bar{f}_{*}\circ h) }{\Ker(h)}
        =\frac{\Ker(h^{'}\circ f_{*}) }{[\pi_{1}(M),\pi_{1}(M)]}\\
		&=\frac{f_{*}^{-1}([\pi_{1}(N),\pi_{1}(N)])}{[\pi_{1}(M),\pi_{1}(M)]}
		=\frac{\Ker(f_{*})\cdot[\pi_{1}(M),\pi_{1}(M)] }{[\pi_{1}(M),\pi_{1}(M)]}\\
		&\cong\frac{\Ker(f_{*})}{\Ker(f_{*})\cap [\pi_{1}(M),\pi_{1}(M)]}
		 =\frac{l_{*}(\pi_{1}(T^{m-k}))}{l_{*}(\pi_{1}(T^{m-k}))\cap [\pi_{1}(M),\pi_{1}(M)]}.
	\end{aligned}\label{eqt3.2}
\end{equation}
The short exact sequence 	
         \begin{equation}
         \xymatrix{
         0 \ar[r] & \Ker(\bar{f}_{*}) \ar[r] & H_{1}(M) \ar[r]^{\bar{f}_{*}} & H_{1}(N) \ar[r] & 0 
         }
         \end{equation}
implies that
\begin{equation}
	\rank(\Ker(\bar{f}_{*}))=\rank(H_{1}(M))-\rank(H_{1}(N))=m-k=\rank(\pi_{1}(T^{m-k})). \label{eqt3.4}
\end{equation}
Combining equalities (\ref{eqt3.2}) and (\ref{eqt3.4}), we obtain the inequality:
\begin{equation}
    \rank(\Ker(\bar{f}_{*}))=m-k\ge \rank(l_{*}(\pi_{1}(T^{m-k})))\ge \rank(\Ker(\bar{f}_{*})),
\end{equation}
which forces equality throughout. Consequently, we have isomorphisms:
\begin{equation}
	\pi_{1}(T^{m-k})\cong l_{*}(\pi_{1}(T^{m-k}))\cong \Ker(\bar{f}_{*}),\label{eqt3.6}
\end{equation}
and two isomorphic maps: 
$ l_{*}:\pi_{1}(T^{m-k})\to l_{*}(\pi_{1}(T^{m-k}))$, $ \text{quot}: l_{*}(\pi_{1}(T^{m-k}))\to \frac{l_{*}(\pi_{1}(T^{m-k}))}{l_{*}(\pi_{1}(T^{m-k}))\cap [\pi_{1}(M),\pi_{1}(M)]}.$

This implies that the image of $\partial_{1}$ is trivial, and the subgroups $l_{*}(\pi_{1}(T^{m-k}))$ and $[\pi_{1}(M),\pi_{1}(M)]$ intersect only at the identity element $e$.
For $ i>1 $, since $ \pi_{i}(T^{m-k}) $ are trivial, we also have $ \im(\partial_{i}) =\{e\}$. This establishes (1) and (2).
 
To prove (3), let $a\in l_{*}(\pi_{1}(T^{m-k}))$ and $b\in \pi_{1}(M)$. Since $ l_{*}(\pi_{1}(T^{m-k}))=\Ker(f_{*}) $ is normal in $ \pi_{1}(M)$, the commutator $[a,b]=a(ba^{-1}b^{-1})$ lies in $l_{*}(\pi_{1}(T^{m-k}))$. By (2), $[a,b]\in l_{*}(\pi_{1}(T^{m-k}))\cap [\pi_{1}(M),\pi_{1}(M)]=\{e\}$. Hence, $a$ commutes with $b$, and $l_{*}(\pi_{1}(T^{m-k}))$ is central.
\end{proof}

We next give an example that satisfies condition (1) of Theorem \ref{Thm3.1}, but fails to satisfy conditions (2),(3), and (\ref{eq:difference}).
\begin{example}
    Let $N_{0}:=S^{3}/\mathbb{Z}_{3}$ be a lens space and consider a representation $\rho:\mathbb{Z}_{3}\to \Aut (T^{3})$ defined by 
    
    \begin{equation}
        \rho(\sigma^{j})=\begin{pmatrix}
                   0&1&0\\ 
                   0&0&1\\ 
                   1&0&0
                    \end{pmatrix}^{j}.
    \end{equation}
Then $M_{0}:=\frac{S^{3}\times T^{3}}{\left \langle(\sigma,~\rho(\sigma))\right \rangle}$ is a torus bundle over $N_{0}$ with structure group $T^{3}\rtimes \Aut (T^{3})$. We verify that the pair $(M_{0},N_{0})$ satisfies condition (1) of Theorem \ref{Thm3.1}, but fails to satisfy condition (2),(3), and (\ref{eq:difference}).
\begin{enumerate}[(a)]
    \item Condition (1) holds because $\pi_{2}(S^{3}/\mathbb{Z}_{3})=\{e\}$ and $\pi_{i}(T^{3})=\{e\}$ for $ i\ge 2$;
    \item Condition (3) fails: The fundamental group is a semidirect product $\pi_{1}(M_{0})=\pi_{1}(T^{3})\rtimes\pi_{1}(N_{0})$ with group law
    \begin{equation}
        (\sigma^{j},t)\cdot(\sigma^{j'},t')=(\sigma^{j+j'},t+\rho(\sigma^{j})(t'))~~~~~\text{for any $j,j'\in \mathbb{Z}$, $t,t \in \mathbb{Z}^{3}$.}
    \end{equation}
    According to Proposition \ref{prop: lifting cri}, by lifting $\rho(\sigma)$ to the universal cover, we obtain
    \begin{equation}
        \pi_{1}(M_{0})\cong \Deck(S^{3}\times \mathbb{R}^{3}/M_{0})=\{(\sigma^{j},\widetilde{\rho(\sigma)}^{j}+\vec{a})|~j\in\mathbb{Z},\vec{a}\in\mathbb{Z}^{3}\},
    \end{equation}
    where $\widetilde{\rho(\sigma)}\in \Aut (\mathbb{Z}^{3})$. The commutator subgroup is
     \begin{equation}
         \begin{aligned}
             [\pi_{1}(M_{0}),\pi_{1}(M_{0})]&\cong \Span \{(\id_{S^{3}},\id_{\mathbb{R}^{3}}+(\widetilde{\rho(\sigma)}-\id)(1,0,0)),\id_{\mathbb{R}^{3}}+(\id_{S^{3}},(\widetilde{\rho(\sigma)}^{2}-\id)(1,0,0))\}\\
             &=\{(\id_{S^{3}},\id_{\mathbb{R}^{3}}+(a,b,c))|~(a,b,c)\in\mathbb{Z}^{3},a+b+c=0\},
         \end{aligned}
     \end{equation}and $H_{1}(M_{0})\cong \mathbb{Z}\oplus\mathbb{Z}_{3}$.
     Hence, $\text{b}_{1}(M_{0})-\text{b}_{1}(N_{0})=1<3$, and 
     \begin{equation}
         l_{*}(\pi_{1}(T^{3}))\cap [\pi_{1}(M_{0}),\pi_{1}(M_{0})]\cong\{(\id_{S^{3}},\id_{\mathbb{R}^{3}}+(a,b,c))|(a,b,c)\in\mathbb{Z}^{3},a+b+c=0\}
     \end{equation} is nontrivial.
\end{enumerate}
\end{example}

\begin{remark}
	In the proof of Theorem \ref{Thm3.1}, the isomorphism (\ref{eqt3.6}) implies that the sequence (\ref{se: short exact sequence (3)}) in homology is exact.
\end{remark}

\begin{lemma}\label{lem3.4}
	Under the assumptions of Theorem \ref{Thm3.1}, if the exact sequence (\ref{se: short exact sequence (3)}) in homology splits, then the sequence (\ref{se: short exact sequence (2)}) in homotopy also splits.
\end{lemma}

\begin{proof}
	Let $ j: H_{1}(N)\to  H_{1}(M) $ be a right inverse of $ \bar{f}_{*} $. Then $ H_{1} (M)= \im(j) \oplus \im(\bar{l}_{*})$.
    
    Let $G:=h^{-1}(\im(j))$. By the correspondence theorem in group theory, $G$ is a normal subgroup of $\pi_{1}(M)$ containing the commutator subgroup $[\pi_{1}(M),\pi_{1}(M)]$. We claim that \begin{equation}
        \pi_{1}(M)=l_{*}(\pi_{1}(T^{m-k}))\times G.
    \end{equation}

    \begin{enumerate}[(1)]
        \item[$\bullet$] Intersection of Subgroups:
        If $a\in l_{*}(\pi_{1}(T^{m-k}))\cap G$, then $h(a)\in \im(\bar{l}_{*})\cap \im(j)=\{e\}$, so $a\in \Ker(h)= [\pi_{1}(M),\pi_{1}(M)]$. By Theorem \ref{Thm3.1} (2), $a=e$;
        \item[$\bullet$] Commutativity: According to Theorem \ref{Thm3.1} (3), every element of $l_{*}(\pi_{1}(T^{m-k}))$ commutes with every element of $G$, so the product is well-defined;
        \item[$\bullet$] Decomposition of $\pi_{1}(M)$:
        For any element $b\in\pi_{1}(M)$, let $h(\text{b}_{1})$ be the projection of $h(b)$ onto $\im(\bar{l}_{*}) $. Then $\text{b}_{1}^{-1}\cdot b\in G$, so $b=\text{b}_{1}\cdot(\text{b}_{1}^{-1}\cdot b)\in l_{*}(\pi_{1}(T^{m-k}))\cdot G$.
    \end{enumerate}
     By the fundamental theorem of group homomorphisms, the restriction $f_{*}|_{G}:G\to\pi_{1}(N)$ is an isomorphism. It follows that the short homotopy exact sequence (\ref{se: short exact sequence (2)}) splits. 
\end{proof}

    If Theorem \ref{the:Topological Splitting Theorem} is assumed, then Corollary \ref{cor 1.3} follows immediately from Lemma \ref{lem3.4}. \\

We now construct the covering spaces $\hat{M}$ and $\hat{N}$ described in Theorem~\ref{the: Structure Theorem}. Let $\text{Tor}(H_{1}(N))$ denote the torsion subgroup of $H_{1}(N)$, and choose a torsion‑free subgroup $\text{TF}(H_{1}(N))$ such that
\begin{equation}
    H_{1}(N)=\text{TF}(H_{1}(N))\oplus \text{Tor}(H_{1}(N)).
\end{equation} 
Since $\text{TF}(H_{1}(N))$ is free abelian, the short exact sequence:
        \begin{equation}
         \xymatrix{
         0 \ar[r] & \Ker(\text{proj}\circ\bar{f}_{*}) \ar[r] & H_{1}(M) \ar[r]^--{\text{proj}\circ\bar{f}_{*}} & \text{TF}(H_{1}(N)) \ar[r] & 0 
         }
         \end{equation}
splits, where $\text{proj}:H_{1}(N)\to \text{TF}(H_{1}(N))$ is a projection. Choose a right-inverse of $\text{proj}\circ\bar{f}_{*}$, and let $K_{1}\subset H_{1}(M)$ be its image. 

Let $K:=h^{-1}(K_{1})\times \im(l_{*})$ (The direct sum decomposition follows by the same argument as in Lemma \ref{lem3.4}.) and $K':=(h')^{-1}(\text{TF}(H_{1}(N)))$. Then:
\begin{equation}
	f_{*}^{-1}(K')=f_{*}^{-1}(h'^{-1}(\text{TF}(H_{1}(N))))=h^{-1}(\bar{f}_{*}^{-1}(\text{TF}(H_{1}(N))))=h^{-1}(K_{1}\times \im(\bar{l}_{*}))=K.
\end{equation}
The direct sum decomposition yields $h^{-1}(K_{1})\cong K^{'}$. Moreover, $K'$ is a normal subgroup of $\pi_{1}(N)$ of finite index, and $K$ is also a normal subgroup of $\pi_{1}(M)$ of the same index (We will next provide the proof in Proposition \ref{prop3.5}(2).). By covering space theory, there exist normal covering spaces $\hat{N}\to N$ and $\hat{M}\to M$ with $\pi_{1}(\hat{N})\cong K'$ and $\pi_{1}(\hat{M})\cong K$, respectively.

The following proposition describes the relationship between these covering spaces.

\begin{proposition}\label{prop3.5}
    The covering spaces $\hat{M}$ and $\hat{N}$ satisfy the following properties:
    \begin{enumerate}[(1)]
         \item The space $\hat{M}$ is also a $T^{m-k}$-bundle over $\hat{N}$ with an affine structure group, and the diagram
         \begin{equation}
             \xymatrix{
             \hat{M} \ar[r]^{\text{pr}_{M}} \ar[d]^{\text{pr}_{\hat{N}}}& M \ar[d]^{f}\\
            \hat{N} \ar[r]^{\text{pr}_{N}}& N
             }
         \end{equation}
          commutes;
        \item $\Deck(\hat{M}/M)\cong \Deck(\hat{N}/N)\cong \text{Tor}(H_{1}(N))$;
        \item The commutator subgroups satisfy $[\pi_{1}(N),\pi_{1}(N)]\subseteq \text{Pr}_{N}(\pi_{1}(\hat{N}))$ and $[\pi_{1}(M),\pi_{1}(M)]\subseteq h^{-1}(K_{1})$;
        \item The homotopy and homology sequences (\ref{se: short exact sequence (2)}) and (\ref{se: short exact sequence (3)}) split for the torus bundle $\hat{M}\to\hat{N}$. 
    \end{enumerate}
\end{proposition}

\begin{proof}
    \begin{enumerate}[(1)]
        \item 
        We prove that $\hat{M}$ is isomorphic to the pullback bundle $\text{pr}_{N}^{*}(M)$ using the universal property.
        
       Consider the diagram:
        \begin{equation}
             \xymatrix{
             \text{pr}_{N}^{*}(M) \ar[r]^-{\text{pr}_{N}^{*}} \ar[d]^{f^{*}}& M \ar[d]^{f}\\
            \hat{N} \ar[r]^-{\text{pr}_{N}}& N
             }
         \end{equation}
        Note that $\text{pr}_{N}^{*}(M)$ is a covering space over $M$. We then show that it is isomorphic to $\hat{M}$ as covering spaces of $M$ by demonstrating that their fundamental groups coincide, each being equal to $K\subseteq \pi_{1}(M)$.
        
         From the homotopy exact sequence of the pullback bundle and the covering map, we obtain the following diagram of fundamental groups:
         
         \begin{equation}
               \xymatrix{
                 \pi_{1}(T^{m-k}) \ar[r]^-{\hat{l_{*}}} \ar[dr]_{l_{*}}   &\pi_{1}(\text{pr}_{N}^{*}(M)) \ar[d]^{(\text{pr}_{N}^{*})_{*}} \ar[r]^-{(f^{*})_{*}}& \pi_{1}(\hat{N}) \ar[d]^{\text{Pr}_{N}}\\
                  & \pi_{1}(M) \ar[r]_-{f_{*}} &\pi_{1}(N).
                 }
         \end{equation}
        By definition, $\text{Pr}_{N}(\pi_{1}(\hat{N}))=K'$. Therefore, 
        \begin{equation}
            f_{*}(\im(\text{pr}_{N}^{*})_{*})=f_{*}\circ (\text{pr}_{N}^{*})_{*}(\pi_{1}(\text{pr}_{N}^{*}(M)))=\text{Pr}_{N}\circ(f^{*})_{*}(\pi_{1}(\text{pr}_{N}^{*}(M)))=K'.
        \end{equation}
        This implies $\im(\text{pr}_{N}^{*})_{*} \subseteq f_{*}^{-1}(K')=K$. Hence, the induced map $(\text{pr}_{N}^{*})_{*}:\pi_{1}(\text{pr}_{N}^{*}(M))\to K$ is well-defined.

         We claim that $(\text{pr}_{N}^{*})_{*}:\pi_{1}(\text{pr}_{N}^{*}(M))\to K$ is an isomorphism: 
         \begin{itemize}
             \item Surjectivity: Let $a\in K$, Since $(f^{*})_{*}$ is surjective and $f_{*}(a)\in K'$, there exists $x\in\pi_{1}(\text{pr}_{N}^{*}(M))$ such that $f_{*}\circ (\text{pr}_{N}^{*})_{*}(x)=\text{Pr}_{N}\circ(f^{*})_{*}(x)=f_{*}(a)$. Then $a\cdot ( \text{pr}_{N}^{*})_{*}(x)^{-1}\in \Ker(f_{*})=\im(l_{*})$. Consequently, we can write $a=(\text{pr}_{N}^{*})_{*}(\hat{l_{*}}(t)\cdot x)$ for some $t\in \pi_{1}(T^{m-k})$.
             \item  Injectivity: Since $\text{pr}^{*}$ is a covering map, $(\text{pr}_{N}^{*})_{*}:\pi_{1}(\text{pr}_{N}^{*}(M))\to \pi_{1}(M)$ is injective.
         \end{itemize}

        By Proposition \ref{prop: lifting cri}, we get an isomorphism of covering spaces $\text{pr}_{N}^{*}(M)\cong \hat{M}$;
        \bigskip
        \item There are natural isomorphisms: 
        \begin{equation}
        \begin{aligned}
            \Deck(\hat{M}/M)&\cong \frac{\pi_{1}(M)}{K}\cong\frac{\pi_{1}(M)/l_{*}(\pi_{1}(T^{m-k}))}{K/l_{*}(\pi_{1}(T^{m-k}))}\cong\frac{\pi_{1}(N)}{K'}\cong \Deck(\hat{N}/N),\\
            \Deck(\hat{N}/N)&\cong \frac{\pi_{1}(N)}{K'}\cong \frac{H_{1}(N)}{\text{TF}(H_{1}(N)) }\cong \text{Tor}(H_{1}(N));\\
        \end{aligned}     
        \end{equation}
        \bigskip
        \item Observe that 
        \begin{equation}
            [\pi_{1}(N),\pi_{1}(N)]=h'^{-1}(0)\subseteq h'^{-1}(\text{TF}(H_{1}(N)))=K'.
        \end{equation} Similarly, we have 
        \begin{equation}
            [\pi_{1}(M),\pi_{1}(M)]=h^{-1}(0)\subseteq h^{-1}(K_{1}).
        \end{equation} Hence both inclusions hold;
        \bigskip
        \item Since $\pi_{1}(\hat{M})\cong K=\im(l_{*})\times h^{-1}(K_{1}) \cong \im(l_{*})\times K_{1}\cong \pi_{1}(T^{m-k})\times \pi_{1}(\hat{N})$, the sequence splits.
    \end{enumerate}
\end{proof}

\begin{remark}
	Note that the first Betti number of $\hat{N}$ may be greater than $\text{b}_{1}(N)$, and $H_{1}(\hat{N})$ may still contain nontrivial torsion elements.

    To illustrate this, we consider a Baumslag–Solitar group 
    \begin{equation}
        \text{BS}(1,n):=\left \langle a,b|a^{-1}ba=b^{n} \right \rangle ~~~~(n\ge 3),
    \end{equation}
    which is finitely presented. By a standard realization theorem (see e.g.\cite[Chapter 5, Chapter 8]{GompfStipsicz}),
    there exists a closed $4$-dimensional manifold $N$ with $\pi_{1}(N)\cong \text{BS}(1,n)$.

    We briefly recall some basic properties of Baumslag–Solitar groups:
    \begin{property}[See~{\cite[Chapter 5]{BS}}]\label{BS group}The following are true:
    \begin{enumerate}[(1)]
        \item Every element of $\text{BS}(1,n)$ can be written in the form $a^{l}(a^{l'}ba^{-l'})^{s}$ with $l,l',s\in\mathbb{Z}$ (here $l$ is unique);
        \item There is a short exact sequence
        \begin{equation}
            \xymatrix{
            0\ar[r]&\mathbb{Z}[\frac{1}{n}]\ar[r]&\text{BS}(1,n) \ar[r]^-{q}& \mathbb{Z}\ar[r]&0
            },
        \end{equation}
        where $q(a^{l}(a^{l'}ba^{-l'})^{s})=l$ and $\mathbb{Z}[\frac{1}{n}]:=\{a^{l'}b^{s}a^{-l'}|~l',s\in \mathbb{Z}\}$ is the localization of the integer ring $\mathbb{Z}$ inside the rational number $\frac{1}{n}$;
        \item Every finite-index subgroup $G$ of $\text{BS}(1,n)$ is generated by two elements of the form
        \begin{equation}
            \left \langle a^{l_{1}}(a^{l'}ba^{-l'})^{s_{1}},b^{s_{2}}\right \rangle~~(l_{1},l',s_{1},s_{2} \in\mathbb{Z},l_{1}>0),
        \end{equation}
       and is isomorphic to a Baumslag-Solitar group $\text{BS}(1,n')$ with $n'\ge 3$;
        \item For $n\ge 3$, the abelianization of $\text{BS}(1,n)$ is isomorphic to $\mathbb{Z}\oplus\mathbb{Z}_{n-1}$.
    \end{enumerate}
    \end{property}
    Now let $\hat{N}$ be any finite‑sheeted covering space of $N$. Then $\pi_{1}(\hat{N})$ is a finite‑index subgroup of $\pi_{1}(N)\cong \text{BS}(1,n)$; by Property~\ref{BS group}(3) it is itself a Baumslag–Solitar group. Consequently, by Property~\ref{BS group}(4), the torsion subgroup $\text{Tor}(H_{1}(\hat{N}))$ is nontrivial. This shows that even after passing to finite covers, the torsion part of the first homology of the base remains nontrivial.
\end{remark}

In Section \ref{sec: Comp}, we will establish a finer relationship between these covering spaces.

\section{The Principal Bundle Structure and Topological Splitting under the Betti Number Condition}\label{sec: Pri}

    In this section, we prove part (1) of Theorem \ref{the: Structure Theorem}, establishing that the torus bundle $M$ is principal under the Betti number condition~(\ref{eq:difference}). The core of the argument lies in showing that the structure group of the bundle reduces to the torus $T^{m-k}$ itself. Having established this principal bundle structure, we then turn to the question of topological triviality. Using obstruction theory, we compute the Euler class of the principal torus bundle appearing in Theorem \ref{the:Topological Splitting Theorem} and complete its proof.

    By Corollary \ref{Red Thm:cor}, it suffices to show that for every loop $\gamma: S^{1}\to N$, the structure group of the pullback bundle $F_{\gamma}:=\gamma^{*}(M)$ admits a reduction to the torus $T^{m-k}$. We begin by describing the topology of such pullback bundles.
       
	 \begin{property} Let $\gamma: S^{1}\to N$ be a loop and $F_{\gamma}:=\gamma^{*}(M)$ the pullback torus bundle over $S^{1}$. Then:
	 	\begin{enumerate}[(1)]
	 		\item $F_{\gamma}$ is bundle isomorphic to the \textbf{mapping torus} $(T^{m-k}\times [0,1])/(z,0)\sim(g_{\gamma}(z),1)$, where $g_{\gamma}\in T^{m-k}\rtimes \Aut (T^{m-k})$.	  
	 		\item  The universal cover of $F_{\gamma}$ is $\mathbb{R}^{m-k+1}$, with deck transformation group
	 		\begin{equation}\label{eqn 4.1}
	 			\Deck(\mathbb{R}^{m-k+1}/F_{\gamma})=\Span\{(\id_{\mathbb{R}^{m-k+1}},e_{1}),(\id_{\mathbb{R}^{m-k+1}},e_{2}),\dots,(\id_{\mathbb{R}^{m-k+1}},e_{m-k}),(\Aff(g_{\gamma}),e_{m-k+1}+\Tran(g_{\gamma}))\},
	 		\end{equation} where $e_{1},\dots,e_{m-k+1}$ are the standard basis vectors of $\mathbb{R}^{m-k+1}$, $\Aff(g_{\gamma})\in \text{GL}(m-k,\mathbb{Z})$ denotes the linear part of $g_{\gamma}$, and $\Tran(g_{\gamma})\in \mathbb{R}^{m-k}\times \{0\}$ is a lift of the translation part of $g_{\gamma}$;
 		\item The generators $(\id_{\mathbb{R}^{m-k+1}},e_{1})$, $(\id_{\mathbb{R}^{m-k+1}},e_{2})$, $\dots$, $(\id_{\mathbb{R}^{m-k+1}},e_{m-k})$ correspond to a basis for $\pi_{1}(T^{m-k})$ of the fiber, while $(\Aff(g_{\gamma}),e_{m-k+1}+\Tran(g_{\gamma}))$ corresponds to a generator of $\pi_{1}$ of the base $S^{1}$.
	 	\end{enumerate}
	 \end{property}

\begin{proof}
    \begin{enumerate}[(1)]
        \item The pullback of the fiber bundle $F_{\gamma}$ via the quotient map $[0,1]\to S^{1}$ is trivial over $[0,1]$, which yields a fiber-preserving surjection $[0,1]\times T^{m-k}\to F_{\gamma}$. The identification of the fibers over $0$ and $1$ is given by an element $g_{\gamma}$ of the affine structure group;
        \bigskip
        \item By part (1), $T^{m-k}\times\mathbb{R}$ is a normal covering space over $F_{\gamma}$, with 
        \begin{equation}
	 	\Deck(T^{m-k}\times\mathbb{R}/F_{\gamma})=\Span\{(g_{\gamma},e_{m-k+1})\},
	    \end{equation}where $(g_{\gamma},e_{m-k+1})(t,r)=(g_{\gamma}(t),r+e_{m-k+1})$ for      $\forall t\in T^{m-k}$ and $\forall r\in\mathbb{R}$.
        
        Using Proposition \ref{prop: lifting cri}, we lift this action to the universal cover over $\mathbb{R}^{m-k+1}$ of $T^{m-k}\times\mathbb{R}$ and obtain a deck transformation group $\tau\in\Deck(\mathbb{R}^{m-k+1}/F_{\gamma})$. We next show that $\tau$ can be written as the form $(\Aff(g_{\gamma}),e_{m-k+1}+\Tran(g_{\gamma}))$.

        Consider the following diagram:
         \begin{equation}
               \xymatrix{
                    \mathbb{R}^{m-k+1}  \ar[d]^{\tau} \ar[rr]^{\text{pr}_{T^{m-k}}\times \id_{\mathbb{R}}}&& T^{m-k}\times \mathbb{R} \ar[d]_{(g_{\gamma},e_{m-k+1})} \ar[drr]^{\text{pr}_{F_{\gamma}}}&&\\
                   \mathbb{R}^{m-k+1} \ar[rr]_{\text{pr}_{T^{m-k}}\times \id_{\mathbb{R}}} && T^{m-k}\times \mathbb{R} \ar[rr]_{\text{pr}_{F_{\gamma}}}&& F_{\gamma},
                 }
         \end{equation}where $\text{pr}_{F_{\gamma}}:T^{m-k}\times \mathbb{R}\to F_{\gamma}$ is the covering map.

          Since $g_{\gamma}\in T^{m-k}\rtimes \Aut (T^{m-k})$, we can write $g_{\gamma}(t)$ as $\Aff(g_{\gamma})t+\vec{a}$, where $\Aff(g_{\gamma})\in \Aut(T^{m-k})\cong GL(m-k,\mathbb{Z})$ and $\vec{a}\in T^{m-k}$. For each $(r_{1},x_{1}),(r_{2},x_{2})\in \mathbb{R}^{m-k}\times \mathbb{R}$, we have:
          \begin{equation}
          \begin{aligned}
                &(\Aff(g_{\gamma})(\text{pr}_{T^{m-k}}(r_{1}))+x_{1}+\vec{a}+e_{m-k+1})\\
                +&(\Aff(g_{\gamma})(\text{pr}_{T^{m-k}}(r_{2}))+x_{2}+\vec{a}+e_{m-k+1})\\
                -&(\Aff(g_{\gamma})(\text{pr}_{T^{m-k}}(r_{1}+r_{2}))+x_{1}+x_{2}+\vec{a}+e_{m-k+1})\\
                -&\vec{a}-e_{m-k+1}=[0]\in T^{m-k}\times\mathbb{R}.
          \end{aligned}
          \end{equation}
          Using the lifting property of $\tau$ and the group homomorphism property of $\text{pr}_{T^{m-k}}\times \id_{\mathbb{R}}$, we obtain 
          \begin{equation}\label{eqn 4.4}
              \tau(r_{1},x_{1})+\tau(r_{2},x_{2})-\tau(r_{1}+r_{2},x_{1}+x_{2})-\Tran(g_{\gamma})-e_{m-k+1}\in \mathbb{Z}^{m-k},
          \end{equation}where $\Tran(g_{\gamma})$ is a lift of $\vec{a}$.
            By the continuity of $\tau$, the left-hand side of (\ref{eqn 4.4}) takes a constant value in $\mathbb{Z}^{m-k}$. Without loss of generality, we may assume that this constant is zero, i.e.,
          \begin{equation}
              \tau(r_{1},x_{1})+\tau(r_{2},x_{2})-\tau(r_{1}+r_{2},x_{1}+x_{2})-\Tran(g_{\gamma})-e_{m-k+1}=0.
          \end{equation} Then the function $\tau-\Tran(g_{\gamma})-e_{m-k+1}$ is linear on $\mathbb{R}^{m-k}$, hence $\tau-\Tran(g_{\gamma})-e_{m-k+1}\in \Aut(\mathbb{R}^{m-k+1})$, we can rewrite $\tau(\cdot)=A(\cdot)+\Tran(g_{\gamma})+e_{m-k+1}$ for some $A\in GL(m-k+1,\mathbb{Z})$. Comparing the linear parts of $\tau$ and $(g_{\gamma},e_{m-k+1})$, we conclude that $A=\Aff(g_{\gamma})$.

         We next verify that the relation (\ref{eqn 4.1}) holds.
          Observe that the set on the right side of (\ref{eqn 4.1}) is contained in the set on the left side. It remains to prove the reverse inclusion.
          For each $\alpha'\in \Deck(\mathbb{R}^{m-k+1}/F_{\gamma})$, we have:
          \begin{equation}
              \text{pr}_{F_{\gamma}}\circ(\text{pr}_{T^{m-k}}\times \id_{\mathbb{R}})\circ \alpha'=\text{pr}_{F_{\gamma}}\circ(\text{pr}_{T^{m-k}}\times \id_{\mathbb{R}}).
          \end{equation}
        By the normality of the covering map $\text{pr}_{F_{\gamma}}$, there exists $\alpha\in \Deck(T^{m-k}\times\mathbb{R}/F_{\gamma})$ such that:
           \begin{equation}
               \alpha\circ(\text{pr}_{T^{m-k}}\times \id_{\mathbb{R}})=(\text{pr}_{T^{m-k}}\times \id_{\mathbb{R}})\circ \alpha'.
           \end{equation}
          Write $\alpha$ as $(g_{\gamma},e_{m-k+1})^{l}$ for some $l\in\mathbb{Z}$. Then:
          \begin{equation}
              \alpha\circ(\text{pr}_{T^{m-k}}\times \id_{\mathbb{R}})=(\text{pr}_{T^{m-k}}\times \id_{\mathbb{R}})\circ (\Aff(g_{\gamma}),e_{m-k+1}+\Tran(g_{\gamma}))^{l}.
          \end{equation}
        By an argument analogous to the one above, we obtain:
        \begin{equation}
            \alpha'-(\Aff(g_{\gamma}),e_{m-k+1}+\Tran(g_{\gamma}))^{l}=\sum_{i=1}^{m-k}k_{i}e_{i}
        \end{equation}for some integers $k_{i}\in\mathbb{Z}$.
       Hence,
       \begin{equation}
            \alpha'=\displaystyle\prod_{i=1}^{m-k}(\id_{\mathbb{R}^{m-k+1}},e_{i})^{k_{i}}\circ(\Aff(g_{\gamma}),e_{m-k+1}+\Tran(g_{\gamma}))^{l}.
        \end{equation} Therefore,
        \begin{equation}
            \alpha'\in \Span\{(\id_{\mathbb{R}^{m-k+1}},e_{1}),(\id_{\mathbb{R}^{m-k+1}},e_{2}),\dots,(\id_{\mathbb{R}^{m-k+1}},e_{m-k}),(\Aff(g_{\gamma}),e_{m-k+1}+\Tran(g_{\gamma}))\};
        \end{equation}
        \bigskip
        \item By Property \ref{property: normal} and Theorem \ref{Thm3.1}, we consider the following two commuting short exact sequences:
        \begin{equation}
            \xymatrix{
            0\ar[r]&\Deck(\mathbb{R}^{m-k+1}/T^{m-k}\times \mathbb{R})\ar[r]\ar@{=}[d]&\Deck(\mathbb{R}^{m-k+1}/F_{\gamma})\ar[r]\ar@{=}[d]&\Deck(T^{m-k}\times \mathbb{R}/F_{\gamma})\ar[r]\ar@{=}[d]&0\\
            0\ar[r]&\pi_{1}(T^{m-k})\ar[r]&\pi_{1}(F_{\gamma})\ar[r]&\pi_{1}(S^{1})\cong \frac{\pi_{1}(F_{\gamma})}{\pi_{1}(T^{m-k})}\ar[r]&0.
            }
        \end{equation}   
        The projection of the path from $0$ to $(e_{m-k+1}+\Tran(g_{\gamma}))$ in $\mathbb{R}^{m-k+1}$ descends to a loop in $F_{\gamma}$ that generates $\pi_{1}(S^{1})$. The remaining generators $(\id,e_{1}),(\id,e_{2}),\dots,(\id,e_{m-k})$ of $\Deck(\mathbb{R}^{m-k+1}/F_{\gamma})$ clearly correspond to the generators of $\pi_{1}(T^{m-k})$. 
    \end{enumerate}
\end{proof}
  
	The next step is to prove that the linear part $\Aff(g_{\gamma})$ equals the identity, which implies that the structure group of $F_{\gamma}$ reduces to pure translations.

     We first check that $\pi_{1}(F_{\gamma})$ is abelian by computing its commutator subgroup.
     \begin{proposition}
         $[\pi_{1}(F_{\gamma}),\pi_{1}(F_{\gamma})]=\{e\}$.
     \end{proposition}
      \begin{proof}
      By Theorem \ref{Thm2.2} and Theorem \ref{Thm3.1}, we consider the following diagram
     \begin{equation}
        \xymatrix{
        0 \ar[r]& \pi_{1}(T^{m-k}) \ar[r]^{l'_{*}}\ar@{=}[d]& \pi_{1}(F_{\gamma}) \ar[r]^{\pi^{*}} \ar[d]^{\gamma^{*}}& \pi_{1}(S^{1}) \ar[r]\ar[d]^{\gamma_{*}}& 0\\
        0 \ar[r]& \pi_{1}(T^{m-k}) \ar[r]_{l_{*}}& \pi_{1}(M) \ar[r]_{f_{*}}& \pi_{1}(N) \ar[r]& 0.
        }
    \end{equation}
    Since $\pi_{1}(S^{1})\cong \mathbb{Z}$ is abelian, we have
    \begin{equation}
        \pi^{*}([\pi_{1}(F_{\gamma}),\pi_{1}(F_{\gamma})])=
    [\pi^{*}(\pi_{1}(F_{\gamma})),\pi^{*}(\pi_{1}(F_{\gamma}))]=\{e\}.
    \end{equation}
    By the exactness of the sequence, we obtain 
    \begin{equation}
         [\pi_{1}(F_{\gamma}),\pi_{1}(F_{\gamma})]\subseteq l'_{*}(\pi_{1}(T^{m-k})).
    \end{equation} 
    Since $l_{*}'$ is injective, its inverse $(l'_{*})^{-1}:\im(l_{*}') \to \pi_{1}(T^{m-k})$ is a well-defined injection. We restrict this inverse to the subgroup $[\pi_{1}(F_{\gamma}),\pi_{1}(F_{\gamma})]\subseteq\im(l_{*}')$.
 
    By the commutativity of the diagram, 
    \begin{equation}
        \gamma^{*}([\pi_{1}(F_{\gamma}),\pi_{1}(F_{\gamma})])\subseteq\pi_{1}(T^{m-k}).
    \end{equation} However, 
     \begin{equation}
        \gamma^{*}([\pi_{1}(F_{\gamma}),\pi_{1}(F_{\gamma})])=[\gamma^{*}(\pi_{1}(F_{\gamma})),\gamma^{*}(\pi_{1}(F_{\gamma}))]\subseteq [\pi_{1}(M),\pi_{1}(M)],
     \end{equation}so
    \begin{equation}
        \gamma^{*}([\pi_{1}(F_{\gamma}),\pi_{1}(F_{\gamma})])\subseteq [\pi_{1}(M),\pi_{1}(M)]\cap \text{Im}(l_{*})=\{e\}.
    \end{equation}
    Since $\gamma^{*}|_{[\pi_{1}(F_{\gamma}),\pi_{1}(F_{\gamma})]}=l_{*}\circ (l'_{*})^{-1}$, the restriction $\gamma^{*}|_{[\pi_{1}(F_{\gamma}),\pi_{1}(F_{\gamma})]}$ is also injective. Consequently, 
    \begin{equation}
         [\pi_{1}(F_{\gamma}),\pi_{1}(F_{\gamma})]=\{e\}.
    \end{equation}
    \end{proof}
  
    By Property~\ref{property: normal}, $\pi_{1}(F_{\gamma})\cong \Deck(\mathbb{R}^{m-k+1}/F_{\gamma})$, hence, $\Deck(\mathbb{R}^{m-k+1}/F_{\gamma})$ is also abelian. This commutativity implies that in the deck transformation group, we have:
	\begin{equation}
		(\id,e_{i})\cdot(\Aff(g_{\gamma}),e_{m-k+1}+\Tran(g_{\gamma}))=(\Aff(g_{\gamma}),e_{m-k+1}+\Tran(g_{\gamma}))\cdot(\id,e_{i}).
	\end{equation}
	Computing both sides shows that $\Aff(g_{\gamma})(e_{i})=e_{i}$ for each $i\in \{1,\dots,m-k\}$. Hence, $\Aff(g_{\gamma})=\id$, and the structure group of $F_{\gamma}$ reduces to $T^{m-k}$.
    
    This completes the proof of Theorem~\ref{the: Structure Theorem}(1).\\

With the principal bundle structure secured, we now address the topological splitting question—namely, when such a bundle is actually trivial. The following proof of Theorem~\ref{the:Topological Splitting Theorem} provides a complete characterization via obstruction theory. We show that under the splitting condition of the homotopy exact sequence, the Euler class vanishes, implying the triviality of the principal torus bundle.
\begin{proof of theorem 1.2} 
     We begin by constructing a cellular decomposition of the base manifold $N$ suitable for obstruction theory. By Whitehead's theorem, any compact manifold is homotopy equivalent to a finite-dimensional CW complex. Let $K=\displaystyle\bigcup_{i=0}^{k}K^{i}$ be a cell complex on $N$ constructed as follows:
     \begin{enumerate}
         \item [$\bullet$] $K^{0}$ is a finite set of points in $N$;
         \item [$\bullet$] For each $i\ge0$, the $(i+1)$-skeleton $K^{i+1}$ is obtained from $K^{i}\displaystyle\coprod_{\alpha\in J_{i+1}}D_{\alpha}^{i+1}$ by identifying $(\alpha,x)$ with $\varphi_{\alpha}(x)$ for $x\in S^{i}\cong\partial D^{i+1}$, where $\varphi_{\alpha}:S^{i}\to K^{i}$ are attaching maps;
        \item [$\bullet$] The index sets $J_{i}=\{\alpha_{1}^{i},\dots,\alpha_{|J_{i}|}^{i}\}$ are finite.
     \end{enumerate}
	 Denote $\rho_{\alpha}^{(i)}:D_{\alpha}^{i}\to K^{i}$ the characteristic maps embedding the cells into the complex, with  $\rho_{\alpha}^{(i)}|_{\partial D_{\alpha}^{i}}=\varphi_{\alpha}$.
     
     We prove that the Euler class $C_{1}(M)$ vanishes using the definition. Let $\sigma_{\text{old}}^{(1)}: K^{1}\to M$ be a section on the $1$-cells of $N$. It suffices to construct a new section $\sigma_{\text{new}}^{(1)}: K^{1}\to M$ on the $1$-cells of $N$ satisfying: 
     
     \begin{enumerate}[(a)]
         \item $\sigma_{\text{new}}^{(1)}|_{K^{0}}=\sigma_{\text{old}}^{(1)}|_{K^{0}}$;
         \item $\sigma_{\text{new}}^{(1)}$ can be extended to $K^{2}$.
         \end{enumerate}

	The section $\sigma_{\text{old}}^{(1)}$ may not have the correct homotopical properties. Let $(l_{*},\zeta):\pi_{1}(T^{m-k})\times \pi_{1}(N)\to \pi_{1}(M)$ be the splitting map from sequence (\ref{se: short exact sequence (2)}). The image of $ (\sigma_{\text{old}}^{(1)})_{*}$ may not lie in $\im(\zeta)$; we need to modify $\sigma_{\text{old}}^{(1)}$ as follows:
    \begin{enumerate}
        \item [$\circ$] Since $ K^{1} $ is a graph, it is homotopy equivalent to $\displaystyle\bigvee_{i=1}^{q}S_{i}^{1}$, so $\pi_{1}(K^{1})$ is a finitely generated free group. Let $\{c_{1},c_{2},\dots,c_{q}\}$ be generators of $\pi_{1}(K^{1})$, each corresponding to a circle in the wedge;
        \item [$\circ$] Let $f_1:K^{1}\to \displaystyle\bigvee_{i=1}^{q}S_{i}^{1}$ and $ f_2:\displaystyle\bigvee_{i=1}^{q}S_{i}^{1}\to  K^{1}$ be homotopy equivalences;
        \item [$\circ$] For each $j=1,2,\dots,q$, choose $ g_{j}:S_{j}^{1}\to T^{m-k}$ such that: 
        \begin{equation}
     	g_{j}+l_{*}^{-1}\circ(\id_{\pi_{1}(M)}-\zeta\circ f_{*})(\sigma_{\text{old}}^{(1)}\circ f_2|_{S_{j}^{1}})=[0].
        \end{equation}
        Then $g_{j}\cdot(\sigma_{\text{old}}^{(1)}\circ f_2|_{S_{j}^{1}})$ lies in $\im(\zeta)$ with $g_{j}$ maps the basepoint of $S^{1}$ to the identity of $T^{m-k}$. Combine these to form $g:\displaystyle\bigvee_{i=1}^{q}S_{i}^{1}\to T^{m-k}$ and define the corrected section: 
        \begin{equation}
     	\sigma_{\text{new}}^{(1)}:=(g\circ f_1) \cdot\sigma_{\text{old}}^{(1)};
         \end{equation}
      \item [$\circ$] Next, we verify that $\im((\sigma_{\text{new}}^{(1)})_{*})$ lies in $\im(\zeta)$.
     For each generator $c_{i}$:
     \begin{equation}
         \sigma_{\text{new}}^{(1)}(c_{i})=(g\circ f_1)(c_{i}) \cdot\sigma_{\text{old}}^{(1)}(c_{i})\sim g_{i} \cdot(\sigma_{\text{old}}^{(1)}\circ f_2|_{S_{i}^{1}}),
     \end{equation}
     which lies in $\im(\zeta)$. Hence, $\im((\sigma_{\text{new}}^{(1)})_{*})\subseteq \im(\zeta)$.
    \end{enumerate}
    
	It is obvious that (a) holds. We next prove (b): For each $2$-cell $e_{\alpha}^{2}$ with characteristic map $\rho_{\alpha}^{(2)}:D_{\alpha}^{2}\to K^{2}$:
     \begin{enumerate}
        \item [$\circ$] Let $\text{proj}_{2}: S^{1}\times[0,1]\to D^{2}$ be the standard quotient map with $\text{proj}_{2}(S^{1}\times\{0\})=\{0\}$;
        \item [$\circ$]  Define $h_{\alpha}^{(2)}:=\rho_{\alpha}^{(2)}\circ \text{proj}_{2}$. Then the restriction $h_{\alpha}^{(2)}(\cdot,1)$ has a lift $\sigma_{\text{new}}^{(1)}\circ \varphi_{\alpha}^{(2)}$;
        \item [$\circ$] By Covering Homotopy Theorem \ref{the:Covering Homotopy Theorem}, there exists a lift $H_{\alpha}^{(2)}$ of $h_{\alpha}^{(2)}$ with $H_{\alpha}^{(2)}(\cdot,0)$ lying in the fiber over $\rho_{\alpha}^{(2)}(0)$. Since $H_{\alpha}^{(2)}(\cdot,0)\in \im(l_{*})\cap \im((\sigma_{\text{new}}^{(1)})_{*})=\{e\}$, the loop $H_{\alpha}^{(2)}(\cdot,0)$ is contractible;
        \item [$\circ$] By the injectivity of $l_{*}$ , there exists a homotopy $ g_{\alpha}^{(2)}:S^{1}\times [0,1]\to T^{m-k}$ from $H_{\alpha}^{(2)}(\cdot,0)$ to a constant map $g_{\alpha}^{(2)}(\cdot,0)\cdot H_{\alpha}^{(2)}(\cdot,0)$;
        \item [$\circ$] Define the extended section over the $2$-cell by quotienting: 
        \begin{equation}
            \widetilde{\rho}_{\alpha}^{(2)}:D_{\alpha}^{2}\to M,~~ \widetilde{\rho}_{\alpha}^{(2)}=\text{the quotient of}~g_{\alpha}^{(2)}\cdot H_{\alpha}^{(2)}.
        \end{equation}
        \end{enumerate}
        
	Define the $2$-skeleton section as: \[\sigma^{(2)}:=
	 \begin{cases}
	 	\sigma_{\text{new}}^{(1)}(x)& x\in K^{1};\\
	 	\widetilde{\rho}_{\alpha}^{(2)}(x)& x\in D_{\alpha}^{2}~.
	 \end{cases} \] After identifying points via the attaching maps, we obtain a section $\sigma^{(2)}: K^{2}\to M$ such that $\sigma^{(2)}|_{K^{1}}=\sigma_{\text{new}}^{(1)}$.

    Our construction yields the vanishing of the primary obstruction $[\mathcal{O}(\sigma_{\text{new}}^{(1)})]\in H^{2}(N;\mathbb{Z}^{m-k})$. By Lemma~\ref{lem2.8}, $[\mathcal{O}(\sigma_{\text{old}}^{(1)})]=0$.
    It follows that $C_{1}(M)=0$, and by Corollary \ref{Cor:2.14}, the principal $T^{m-k}$-bundle $M\to N$ is trivial. This completes the proof of Theorem~\ref{the:Topological Splitting Theorem}.    
     \end{proof of theorem 1.2}
     
\section{Classification of Torus Bundles via Deck Transformations}\label{sec: Comp}
    In this section, we classify all torus bundles $M$ appearing in Theorem~\ref{the: Structure Theorem} by analyzing the deck transformation groups of the covering spaces constructed in Section \ref{sec: Homo}. We establish the explicit structure of these groups and prove the converse statement (see Theorem~\ref{Thm 5.7}).
	
    Following the results of Section \ref{sec: Homo} and Theorem~\ref{the:Topological Splitting Theorem}, we obtain finite normal covering spaces $\hat{M}$ of $M$ and $\hat{N}$ of $N$ such that $\hat{M}$ is diffeomorphic to $\hat{N}\times T^{m-k}$. We now analyze the deck transformation group $\Deck(\hat{M}/M)$.

  \begin{theorem}
      Under the assumption of Theorem~\ref{the: Structure Theorem}, 
      the deck transformation group $\Deck(\hat{M}/M)$ is isomorphic to the group consisting of all transformations of the form
    	 \begin{equation}
    	 	\Deck(\hat{M}/M)=\{(\alpha,\text{pr}_{T^{m-k}}\circ T_{\alpha}+\id_{T^{m-k}})|~\alpha \in \Deck(\hat{N}/N)\},
    	 	\end{equation}  
            where the maps $T_{\alpha}:\hat{N}\to\mathbb{R}^{m-k}$ are smooth and satisfy the following cocycle and commutativity conditions for all $\alpha,\beta\in \Deck(\hat{N}/N)$ and $x\in\hat{N}$:
    	 	\begin{equation}\label{group conditions}
            \begin{aligned}
               \text{pr}_{T^{m-k}}\circ T_{\alpha\cdot\beta}(x)&=\text{pr}_{T^{m-k}}\circ T_{\alpha}(\beta(x))+\text{pr}_{T^{m-k}}\circ T_{\beta}(x)~~~(\text{cocycle condition});\\
               T_{\alpha}(\beta(x))-T_{\alpha}(x)&=T_{\beta}(\alpha(x))-T_{\beta}(x)~~~(\text{commutativity condition}).	
            \end{aligned}
    	 	\end{equation}    
  \end{theorem}
  To further analyze the structure of $\Deck(\hat{M}/M)$, we first establish its compatibility with the global torus action.
    \begin{lemma}\label{lem5.2}
        The global torus action on $\hat{M}$ commutes with every element of $\Deck(\hat{M}/M)$.
    \end{lemma}
    \begin{proof}
        For any $g\in T^{m-k}$ and $\hat{\alpha}\in \Deck(\hat{M}/M)$, consider the conjugation:
        \begin{equation}
            \text{Ad}(g)(\hat{\alpha}):=g^{-1}\circ\hat{\alpha}\circ g.
        \end{equation}
        Let $\text{pr}_{M}:\hat{M}\to M$ be the covering map. For any $(y,g')\in \hat{N}\times T^{m-k}\cong\hat{M}$:
        \begin{equation}
        \begin{aligned}
             \text{pr}_{M}(\text{Ad}(g)(\hat{\alpha})(y,g'))&=\text{pr}_{M}(g^{-1}\circ\hat{\alpha}\circ g(y,g'))\\
             &=g^{-1}\cdot \text{pr}_{M}(\hat{\alpha}\circ g(y,g'))\\
             &=g^{-1}\cdot \text{pr}_{M}(g(y,g'))\\
             &=g^{-1}\cdot g\cdot \text{pr}_{M}(y,g')=\text{pr}_{M}(y,g').
        \end{aligned}  
        \end{equation}

   By the continuity of $\text{Ad}(g)(\hat{\alpha})$ and the connectedness of $\hat{N}\times T^{m-k}$, it follows that
   $\text{Ad}(g)(\hat{\alpha})\in \Deck(\hat{M}/M)$. Thus, we obtain a continuous homomorphism:
   \begin{equation}
       \text{Ad}: T^{m-k}\to \Aut (\Deck(\hat{M}/M)),~~~\text{Ad}(g):(\hat{\alpha}\mapsto g^{-1}\circ\hat{\alpha}\circ g).
   \end{equation}
 Since $\Aut (\Deck(\hat{M}/M))$ is discrete, $\text{Ad}$ must be constant. Evaluating at the identity $g=[0]$ gives $\text{Ad}(g)(\hat{\alpha})=\hat{\alpha}$ for all $g\in T^{m-k}$. Therefore,
   \begin{equation}
       g\circ\hat{\alpha}=g\circ \text{Ad}(g)(\hat{\alpha})=g\circ g^{-1}\circ\hat{\alpha}\circ g=\hat{\alpha}\circ g~,
   \end{equation}
   which shows that the actions commute.
   \end{proof}
   
   We now exploit the product structure $\hat{M}\cong \hat{N}\times T^{m-k}$.
    \begin{lemma}\label{lem5.3}
    The deck transformation $\hat{\alpha}$ has the following form:
    \begin{equation}
        \hat{\alpha}(x,g)=(\alpha(x),T'_{\hat{\alpha}}(x)+g),~~~\forall x\in\hat{N},g\in T^{m-k},
    \end{equation}
       where $\alpha\in \Deck(\hat{N}/N)$ and $T'_{\hat{\alpha}}:\hat{N}\to T^{m-k}$ is a smooth map.
    \end{lemma} 
    \begin{proof}
        Let $\hat{\alpha}(x,[0])=(\alpha(x),T'_{\hat{\alpha}}(x))$. Since $\hat{\alpha}$ commutes with the $T^{m-k}$-action by Lemma \ref{lem5.2}, we have:   
        \begin{equation}
        \hat{\alpha}(x,g)=g\circ\hat{\alpha}(x,[0])=(\alpha(x),T'_{\hat{\alpha}}(x)+g).
    \end{equation}
    
    Next, we verify that the first projection $\alpha$ must be a deck transformation on $\hat{N}$.

    Similarly, we can write $\hat{\alpha}^{-1}(x,g)$ as $(\beta(x),T'_{\hat{\alpha}^{-1}}(x)+g)$. The identity $\hat{\alpha}^{-1}\circ\hat{\alpha}(x,g)=(x,g)$ yields $\beta\circ\alpha(x)=x$, hence $\alpha$ is a diffeomorphism.
    
    To see that $\alpha\in \Deck(\hat{N}/N)$, we verify that $\text{pr}_{N}\circ\alpha=\text{pr}_{N}$. A direct calculation gives
    \begin{equation}
    \begin{aligned}
        \text{pr}_{N}(\alpha(x))&=\text{pr}_{N}\circ \text{pr}_{\hat{N}}(\alpha(x),T'_{\hat{\alpha}}(x))\\
        &=f\circ \text{pr}_{M}(\alpha(x),T'_{\hat{\alpha}}(x))\\
        &=f\circ \text{pr}_{M}\circ\hat{\alpha}(x,[0])\\
        &=f\circ \text{pr}_{M}(x,[0])\\
        &=\text{pr}_{N}\circ \text{pr}_{\hat{N}}(x,[0])=\text{pr}_{N}(x),
    \end{aligned}
    \end{equation}
   where $\text{pr}_{\hat{N}}:\hat{M}\to\hat{N}$ is the first component projection. Thus, $\alpha\in \Deck(\hat{N}/N)$.
    \end{proof}
    
	The assignment $\hat{\alpha}\mapsto\alpha$ defines a group homomorphism:
     \begin{equation}
         De:\Deck(\hat{M}/M)\to \Deck(\hat{N}/N).
     \end{equation}
    \begin{lemma}\label{lem5.4}
        The homomorphism $De$ is an isomorphism.
    \end{lemma}
    \begin{proof}
        We show that $De$ is surjective.
 For any $\alpha_{0}\in \Deck(\hat{N}/N)$ and $y_{0}\in \hat{N}$, the principal bundle structure ensures that 
    \begin{equation}
        f\circ \text{pr}_{M}(y_{0},[0])=\text{pr}_{N}(y_{0})=\text{pr}_{N}(\alpha_{0}(y_{0}))=f\circ \text{pr}_{M}(\alpha_{0}(y_{0}),[0])
    \end{equation}   
    Thus, $\text{pr}_{M}(y_{0},[0])$ lies in the same $T^{m-k}$-orbit as $\text{pr}_{M}(\alpha_{0}(y_{0}),[0])$. The normality of the covering space implies that there exists $\hat{\alpha}_{0}\in \Deck(\hat{M}/M)$ such that 
    \begin{equation}
        \hat{\alpha}_{0}(y_{0},[0])=(\alpha_{0}(y_{0}),T'_{\hat{\alpha}_{0}}(y_{0})).
    \end{equation}
    Thus, $De$ is surjective. Together with the isomorphism $\Deck(\hat{N}/N)\cong \Deck(\hat{M}/M)$ in Proposition~\ref{prop3.5}(2), it follows that $De$ is an isomorphism.
    \end{proof}
    
     Denote $T'_{\alpha}:=T'_{De^{-1}(\alpha)}$. We obtain
    \begin{equation}
        \Deck(\hat{M}/M)=\{(\alpha,T'_{\alpha}+\id_{T^{m-k}})|~\alpha\in \Deck(\hat{N}/N)\}.
    \end{equation}

  By Proposition~\ref{prop3.5}(3), the image of $\pi_{1}(\hat{N}\times \mathbb{R}^{m-k})$ under the covering map is the normal subgroup $h^{-1}(K_{1})$ in $\pi_{1}(M)$. So, it follows from Property \ref{property: normal} that $\hat{N}\times \mathbb{R}^{m-k}$ is a normal covering space of $M$.

\begin{lemma}\label{lem6.4}
    There exists a family of smooth maps $T_{\alpha}:\hat{N}\to \mathbb{R}^{m-k}$ indexed by $\alpha\in \Deck(\hat{N}/N)$ such that
    \begin{equation}\label{eqn 5.14}
          \Deck(\hat{N}\times \mathbb{R}^{m-k}/M)=\{(\alpha,T_{\alpha}+\vec{a}+\id_{\mathbb{R}^{m-k}})|~\alpha\in \Deck(\hat{N}/N),\vec{a}\in \mathbb{Z}^{m-k}\},
    \end{equation}
    and $\text{pr}_{T^{m-k}}\circ T_{\alpha}=T_{\alpha}'$.
\end{lemma}

 \begin{proof}
     Using the induced $\mathbb{R}^{m-k}$-action on the normal covering space $\hat{N}\times \mathbb{R}^{m-k}$ and an argument analogous to those in Lemmas \ref{lem5.2}-\ref{lem5.4}, we can decompose any deck transformation $\hat{\alpha}'\in \Deck(\hat{N}\times\mathbb{R}^{m-k}/M)$ as $(\alpha,T''+\vec{a}+\id_{\mathbb{R}^{m-k}})$, where $\alpha \in \Deck(\hat{N}/N)$, $\vec{a} \in \mathbb{Z}^{m-k}$, and a map $T'':\hat{N}\to\mathbb{R}^{m-k}$. Denoting the map $T''$ by $T_{\alpha}$, we obtain the inclusion
     \begin{equation}
          \Deck(\hat{N}\times \mathbb{R}^{m-k}/M)\subseteq\{(\alpha,T_{\alpha}+\vec{a}+\id_{\mathbb{R}^{m-k}})|~\alpha\in \Deck(\hat{N}/N), T_{\alpha}:\hat{N}\to \mathbb{R}^{m-k},\vec{a}\in \mathbb{Z}^{m-k}\}.
     \end{equation}
     
     We next verify that
     \begin{equation}\label{eqn 5.16}
                \Deck(\hat{N}\times\mathbb{R}^{m-k}/M)\supseteq\{(\alpha,T_{\alpha}+\vec{a}+\id_{\mathbb{R}^{m-k}})|~\alpha\in \Deck(\hat{N}/N), T_{\alpha}:\hat{N}\to \mathbb{R}^{m-k},\vec{a}\in \mathbb{Z}^{m-k}\}.
             \end{equation} 
             
    For each $(y_{0},r_{0})\in \hat{N}\times\mathbb{R}^{m-k}$ and $(\alpha,T_{\alpha}+\vec{a}+\id_{\mathbb{R}^{m-k}})$ with $(\alpha,T_{\alpha}+\vec{a'}+\id_{\mathbb{R}^{m-k}})\in \Deck(\hat{N}\times \mathbb{R}^{m-k}/M)$ for some $\vec{a'}\in \mathbb{Z}^{m-k}$:
           \begin{equation}
           \label{eqn 5.17}
           \begin{aligned}
               &\quad \text{pr}_{M}\circ(\id_{\hat{N}}\times \text{pr}_{T^{m-k}})((\alpha,T_{\alpha}+\vec{a}+\id_{\mathbb{R}^{m-k}})\cdot
               (y_{0},r_{0}))\\
               &=\text{pr}_{M}(\alpha(y_{0}),\text{pr}_{T^{m-k}}(T_{\alpha}(y_{0})+\vec{a}+r_{0}))
               =\text{pr}_{M}(\alpha(y_{0}),\text{pr}_{T^{m-k}}(T_{\alpha}(y_{0})+\vec{a'}+r_{0}))\\
               &=\text{pr}_{M}\circ(\id_{\hat{N}}\times \text{pr}_{T^{m-k}})((\alpha,T_{\alpha}+\vec{a'}+\id_{\mathbb{R}^{m-k}})\cdot
               (y_{0},r_{0}))\\
               &=\text{pr}_{M}\circ(\id_{\hat{N}}\times \text{pr}_{T^{m-k}})(y_{0},r_{0}).
           \end{aligned}  
           \end{equation}

       By the normality of the covering space, the points $(\alpha,T_{\alpha}+\vec{a}+\id_{\mathbb{R}^{m-k}})\cdot (y_{0},r_{0})$ and $(y_{0},r_{0})$ lie in the same $\Deck(\hat{N}\times \mathbb{R}^{m-k}/M)$-orbit, hence (\ref{eqn 5.16}) holds.

       Equality (\ref{eqn 5.17}) implies that the images of the two points under $(\id_{\hat{N}}\times \text{pr}_{T^{m-k}})$ belong to the same $\Deck(\hat{M}/M)$-orbit. Consequently, there exists $\alpha_{0}\in \Deck(\hat{M}/M)$ such that 
               \begin{equation}
                   (\id_{\hat{N}}\times \text{pr}_{T^{m-k}})((\alpha,T_{\alpha}+\vec{a}+\id_{\mathbb{R}^{m-k}})\cdot
               (y_{0},r_{0}))=\alpha_{0}\circ(\id_{\hat{N}}\times \text{pr}_{T^{m-k}})(y_{0},r_{0}).
               \end{equation}
       Comparing the first components gives $De(\alpha_{0})=\alpha$; comparing the second components yields $\text{pr}_{T^{m-k}}\circ T_{\alpha}=T_{\alpha}'$.
 \end{proof}

 It is straightforward to verify that $\{T_{\alpha}\}_{\alpha\in \Deck(\hat{N}/N)}$ satisfies the cocycle condition (\ref{group conditions}). However, to get the commutativity condition, we need the following lemma:
  \begin{lemma}\label{lem5.6} The deck transformation group $\Deck(\hat{N}\times \mathbb{R}^{m-k}/M)$ is abelian.
  \end{lemma}
  \begin{proof}
  
      By Property \ref{property: normal}, we have the commutative diagram:
      
      \begin{equation}
          \xymatrix{
          0\ar[dr]&&0\ar[d]&0\ar[d]&\\
          &\pi_{1}(\hat{M})\cong\pi_{1}(\hat{N})\times\pi_{1}(T^{m-k})\ar[dr]^{\text{Pr}_{M}}&\pi_{1}(T^{m-k})\ar[d]^{l_{*}}\ar@{=}[r]&\Deck(\hat{N}\times \mathbb{R}^{m-k}/\hat{M})\ar[d]^{\text{Pr}_{\hat{M}/M}}&\\
          0\ar[r]&\pi_{1}(\hat{N}\times\mathbb{R}^{m-k})\ar[r]_-{\text{Pr}_{\hat{N}\times\mathbb{R}^{m-k}/M}}\ar@{=}[d]&\pi_{1}(M)\ar[dr]^{\text{Pr}_{M}^{*}}\ar[r]^-{\text{Pr}_{\hat{N}\times\mathbb{R}^{m-k}/M}^{*}}\ar[d]^{f_{*}}&\Deck(\hat{N}\times \mathbb{R}^{m-k}/M)\ar[d]^{(\text{Pr}_{T^{m-k}})_{*}}\ar[r]&0\\
          0\ar[r]&\pi_{1}(\hat{N})\ar[r]_-{\text{Pr}_{N}}&\pi_{1}(N)\ar[r]_-{\text{Pr}_{N}^{*}}\ar[d]&\Deck(\hat{M}/M)\cong \Deck(\hat{N}/N)\ar[dr]\ar[r]\ar[d]&0\\
          &&0&0&0.
          }
      \end{equation}

        We observe that
        \begin{equation}          
            \text{Pr}_{M}=\text{Pr}_{\hat{N}\times \mathbb{R}^{m-k}/M}\circ \text{Pr}_{\hat{N}}+l_{*}\circ \text{Pr}_{T^{m-k}}.      
        \end{equation}
        By Proposition~\ref{prop3.5}(3), $[\pi_{1}(M),\pi_{1}(M)]\subseteq h^{-1}(K_{1})=\im(\text{Pr}_{\hat{N}\times\mathbb{R}^{m-k}/M})$. It follows that
        \begin{equation}
        \begin{aligned}
              \quad[\Deck(\hat{N}\times \mathbb{R}^{m-k}/M),\Deck(\hat{N}\times \mathbb{R}^{m-k}/M)]
              &=\text{Pr}^{*}_{\hat{N}\times\mathbb{R}^{m-k}/M}([\pi_{1}(M),\pi_{1}(M)])\\
              &\subseteq \text{Pr}^{*}_{\hat{N}\times\mathbb{R}^{m-k}/M}(\im(\text{Pr}_{\hat{N}\times\mathbb{R}^{m-k}/M}))=\{e\}.
              \end{aligned}    
              \end{equation}
              Hence, $\Deck(\hat{N}\times \mathbb{R}^{m-k}/M)$ is abelian.
        \end{proof}
  
    By Lemma \ref{lem5.6}, the transformations commute:
     \begin{equation}
         (\alpha,T_{\alpha}+\vec{a}+\id_{\mathbb{R}^{m-k}})\cdot (\beta,T_{\beta}+\vec{b}+\id_{\mathbb{R}^{m-k}})=(\beta,T_{\beta}+\vec{b}+\id_{\mathbb{R}^{m-k}})\cdot (\alpha,T_{\alpha}+\vec{a}+\id_{\mathbb{R}^{m-k}})
     \end{equation}
        This implies the commutativity condition:
         \begin{equation}
              T_{\alpha}(\beta(x))-T_{\alpha}(x)=T_{\beta}(\alpha(x))-T_{\beta}(x),~~\forall \alpha,\beta\in\Deck(\hat{N}/N),x\in \hat{N}.
         \end{equation}
         \\

   Having described the structure of deck transformation groups for bundles arising from Theorem~\ref{the: Structure Theorem}, we now prove the converse: given the data $(N,\hat{N},\{T_{\alpha}\}_{\alpha\in \Deck(\hat{N}/N)})$ satisfying the stated hypotheses, the resulting manifold $M$ is a principal torus bundle satisfying the Betti number equality.

    \begin{theorem}\label{Thm 5.7}Let $N^{k}$ be a closed manifold and $\hat{N}$ a finite normal covering space of $N$ with an abelian deck transformation group $\Deck(\hat{N}/N)$. Suppose we are given a family of smooth maps $T_{\alpha}:\hat{N}\to \mathbb{R}^{m-k}$ indexed by $\alpha\in \Deck(\hat{N}/N)$
    satisfying for all $ \alpha,\beta\in \Deck(\hat{N}/N)$ and $x\in \hat{N}$:
           \begin{equation}\label{eqn 5.24}    
               \text{pr}_{T^{m-k}}\circ T_{\alpha\cdot\beta}(x)=\text{pr}_{T^{m-k}}\circ T_{\alpha}(\beta(x))+\text{pr}_{T^{m-k}}\circ T_{\beta}(x)~~~(\text{cocycle condition}).	 
    	 	\end{equation} 
       Then, the quotient $m$-dimensional manifold 
       \begin{equation}
           M:=(\hat{N}\times T^{m-k})/\{(\alpha,\text{pr}_{T^{m-k}}\circ T_{\alpha}+\id_{T^{m-k}})|~\alpha\in \Deck(\hat{N}/N)\}
       \end{equation} satisfies the following:
         \begin{enumerate}[(1)]
             \item $M$ is a principal $T^{m-k}$-bundle over $N$ with a smooth $T^{m-k}$-action;
             \item The family of smooth maps $\{T_{\alpha}\}_{\alpha\in \Deck(\hat{N}/N)}$ satisfies the following commutativity condition:
             \begin{equation}
                   T_{\alpha}(\beta(x))-T_{\alpha}(x)=T_{\beta}(\alpha(x))-T_{\beta}(x)~~~(\text{commutativity condition}) 
             \end{equation} for all $ \alpha,\beta\in \Deck(\hat{N}/N)$ and $x\in \hat{N}$;
             \item The covering space $\hat{N}\times\mathbb{R}^{m-k}$ of $M$ is normal, and the group $\Deck(\hat{N}\times\mathbb{R}^{m-k}/M)$ is abelian;
             \item The homotopy and homology short exact sequences (\ref{se: short exact sequence (2)}) and (\ref{se: short exact sequence (3)}) hold;
             \item $M$ and $N$ satisfy the Betti number condition (\ref{eq:difference}).
         \end{enumerate}         
    \end{theorem}
      \begin{proof}
          \begin{enumerate}[(1)]
              \item For any $\alpha\in \Deck(\hat{N}/N)$, $y\in\hat{N}$ and $g,g'\in T^{m-k}$, we compute
              \begin{equation}
              \begin{aligned}
                  (\alpha,\text{pr}_{T^{m-k}}\circ T_{\alpha}+\id_{T^{m-k}})\cdot g\cdot (y,g')&=(\alpha,\text{pr}_{T^{m-k}}\circ T_{\alpha}+\id_{T^{m-k}})\cdot (y,g+g')\\
                  &=(\alpha(y),\text{pr}_{T^{m-k}}\circ T_{\alpha}(y)+g+g')\\
                  &=g\cdot(\alpha(y),\text{pr}_{T^{m-k}}\circ T_{\alpha}(y)+g')\\
                  &=g\cdot(\alpha,\text{pr}_{T^{m-k}}\circ T_{\alpha}+\id_{T^{m-k}})\cdot (y,g'),
              \end{aligned} 
              \end{equation} which shows that the torus action commutes with $\Deck(\hat{N}\times T^{m-k}/M)$.
              
            We define a smooth fibration $f:M\to N$ by
            $f(\text{pr}_{M}(y,g')):=\text{pr}_{N}(y)$ for $y\in\hat{N}$ and $g'\in T^{m-k}$. This map is well-defined. Moreover, the commutativity established above guarantees that $M$ is a principal $T^{m-k}$-bundle over $N$;
            \bigskip
            \item For all $\alpha,\beta\in \Deck(\hat{N}/N)$, the commutativity of $\Deck(\hat{N}/N)$ and the cocycle condition imply:
            \begin{equation}
                \text{pr}_{T^{m-k}}\circ T_{\beta}(\alpha(x))+\text{pr}_{T^{m-k}}\circ T_{\alpha}(x)=\text{pr}_{T^{m-k}}\circ T_{\alpha}(\beta(x))+\text{pr}_{T^{m-k}}\circ T_{\beta}(x).
            \end{equation}
            Hence, the function $T_{\beta}(\alpha(x))+T_{\alpha}(x)-T_{\alpha}(\beta(x))-T_{\beta}(x)$ takes values in $\mathbb{Z}^{m-k}$. By the continuity of $\{T_{\alpha}\}$, there exists a constant $\vec{a}_{\alpha,\beta}\in\mathbb{Z}^{m-k}$ such that
            \begin{equation}\label{eqn 5.29}
               T_{\beta}(\alpha(x))+T_{\alpha}(x)-T_{\alpha}(\beta(x))-T_{\beta}(x)=\vec{a}_{\alpha,\beta}
            \end{equation} for any $x\in \hat{N}$.

          Suppose that $\alpha^{L_{1}}=\beta^{L_{2}}=\id_{\hat{N}}$. Fix a basepoint $x_{0}\in\hat{N}$. For any $x'\in\hat{N}$, consider the points $x',\alpha(x'),\dots ,\alpha^{L_{1}-1}(x')$. Substituting these into equation (\ref{eqn 5.29}) and summing over $j=0$ to $L_{1}-1$ yields
          \begin{equation}\label{eqn 5.30}
              \displaystyle\sum_{j=0}^{L_{1}-1}[T_{\alpha}(\alpha^{j}(x'))-T_{\alpha}(\beta\cdot\alpha^{j}(x'))]=L_{1}\vec{a}_{\alpha,\beta},~~\forall x'\in\hat{N}.
          \end{equation}
          Now take $x'=x_{0},\beta(x_{0}),\dots ,\beta^{L_{2}-1}(x_{0})$. Substituting each of these into equation $(\ref{eqn 5.30})$ and summing over $l=0$ to $L_{2}-1$ gives
          \begin{equation} 0=\displaystyle\sum_{l=0}^{L_{2}-1}\displaystyle\sum_{j=0}^{L_{1}-1}[T_{\alpha}(\beta^{l}\cdot\alpha^{j}(x_{0}))-T_{\alpha}(\beta^{l+1}\cdot\alpha^{j}(x_{0}))]=L_{2}L_{1}\vec{a}_{\alpha,\beta}.
          \end{equation} 
         Therefore, $\vec{a}_{\alpha,\beta}=0$, and the commutativity condition follows; 
         \bigskip
            \item Let $\text{pr}_{M}:\hat{N}\times T^{m-k}\to M$ be
            the covering (quotient) map. By Property~\ref{property: normal}(2)(a), $\text{pr}_{M}$ is a normal map. We first verify that $\hat{N}\times \mathbb{R}^{m-k}$ is a normal covering space of $M$.

            For any $x\in M$ and any $(y_{1},g_{1}),(y_{2},g_{2})\in (\text{pr}_{M}\circ(\id_{\hat{N}}\times \text{pr}_{T^{m-k}}))^{-1}(x)$, the normality of the covering space $\hat{N}\times T^{m-k}$ implies the existence of an element $(\alpha,\text{pr}_{T^{m-k}}\circ T_{\alpha}+\id_{T^{m-k}})$ such that $(\alpha,\text{pr}_{T^{m-k}}\circ T_{\alpha}+\id_{T^{m-k}})(y_{1},\text{pr}_{T^{m-k}}(g_{1}))=(y_{2},\text{pr}_{T^{m-k}}(g_{2}))$. By Proposition \ref{prop: lifting cri}, there exists a lift $\alpha'':\hat{N}\times \mathbb{R}^{m-k}\to \hat{N}\times \mathbb{R}^{m-k}$ of $(\alpha,\text{pr}_{T^{m-k}}\circ T_{\alpha}+\id_{T^{m-k}})$ satisfying:
            \begin{equation}
            \begin{aligned}
            \alpha''(y_{1},g_{1})&=(y_{2},g_{2}),\\
                (\id_{\hat{N}}\times \text{pr}_{T^{m-k}})\circ \alpha''&=(\alpha,\text{pr}_{T^{m-k}}\circ T_{\alpha}+\id_{T^{m-k}})\circ(\id_{\hat{N}}\times \text{pr}_{T^{m-k}}).
            \end{aligned}  
            \end{equation}
             Thus, $\text{pr}_{M}\circ(\id_{\hat{N}}\times \text{pr}_{T^{m-k}})\circ\alpha''=\text{pr}_{M}\circ(\id_{\hat{N}}\times \text{pr}_{T^{m-k}})$ and  $\alpha''\in \Deck(\hat{N}\times\mathbb{R}^{m-k}/M)$, confirming that $\hat{N}\times \mathbb{R}^{m-k}$ is a normal covering of $M$. 
             
             We next show that:
             \begin{equation}\label{eqn 5.33}
                 \Deck(\hat{N}\times\mathbb{R}^{m-k}/M)=\{(\alpha,T_{\alpha}+\vec{a}+\id_{\mathbb{R}^{m-k}})|~\alpha\in \Deck(\hat{N}/N),\vec{a}\in \mathbb{Z}^{m-k}\}.
             \end{equation}
            Lemma \ref{lem6.4} provides a family of smooth maps $TT_{\alpha}:\hat{N}\to \mathbb{R}^{m-k}$ indexed by $\alpha\in \Deck(\hat{N}/N)$ satisfying (\ref{eqn 5.14}) with $\text{pr}_{T^{m-k}}\circ TT_{\alpha}=\text{pr}_{T^{m-k}}\circ T_{\alpha}$. Thus, $TT_{\alpha}-T_{\alpha}$ is constant in $\mathbb{Z}^{m-k}$; after replacing $TT_{\alpha}$ by $T_{\alpha}$, we obtain the desired form (\ref{eqn 5.33}).
          
            Finally, the commutativity condition forces $\Deck(\hat{N} \times \mathbb{R}^{m-k}/ M)$ to be abelian;
          \bigskip
         \item Fix a base point $y_{0}\in \hat{N}$. Define a fiber embedding map $l:T^{m-k}\to M$ by $l(g):=\text{pr}_{M}(y_{0},g)$ for any $g\in T^{m-k}$. Let  $l_{y_{0}}:T^{m-k}\to \hat{N}\times T^{m-k}$ be given by $l_{y_{0}}(g):=(y_{0},g)$. Then, $l=\text{pr}_{M}\circ l_{y_{0}}$, $l_{*}$ is an injective homomorphism:
            \begin{equation}
                l_{*}:\pi_{1}(T^{m-k})\xhookrightarrow{(l_{y_{0}})_{*}}\pi_{1}(\hat{N}\times T^{m-k})\xhookrightarrow{\text{Pr}_{M}}\pi_{1}(M).
            \end{equation}
         Combining  this with Theorem \ref{Thm2.2}, we get $\im(\partial_{1})=\Ker(l_{*})=\{e\}$, yielding the short exact sequence (\ref{se: short exact sequence (2)}) of fundamental groups.
         
         Consider the exact sequence of covering spaces:
       \begin{equation}
           0\xrightarrow{} \pi_{1}(\hat{N}\times\mathbb{R}^{m-k}) \xrightarrow{\text{Pr}_{\hat{N}\times\mathbb{R}^{m-k}/M}}\pi_{1}(M) \xrightarrow{\text{Pr}_{\hat{N}\times\mathbb{R}^{m-k}/M}^{*}}
           \Deck(\hat{N}\times\mathbb{R}^{m-k}/M) \xrightarrow{} 0~.
       \end{equation}
        By part (3), for any $a,b\in \pi_{1}(M)$,
           \begin{equation}
                \text{Pr}_{\hat{N}\times\mathbb{R}^{m-k}/M}^{*}[a,b]=[\text{Pr}_{\hat{N}\times\mathbb{R}^{m-k}/M}^{*}(a),\text{Pr}_{\hat{N}\times\mathbb{R}^{m-k}/M}^{*}(b)]\\
                =\id.  
           \end{equation}
             Hence, $[\pi_{1}(M),\pi_{1}(M)]$ lies in the image of $\pi_{1}(\hat{N})$, and therefore 
             $l_{*}(\pi_{1}(T^{m-k}))\cap [\pi_{1}(M),\pi_{1}(M)]=\{e\}.$
             Together with diagram (\ref{big diagram}), this yields the short exact sequence (\ref{se: short exact sequence (3)}) in homology;
          \bigskip   
      \item From part (4), we deduce:
              \begin{equation}
                  \text{b}_{1}(M)-\text{b}_{1}(N)=\text{b}_{1}(T^{m-k})=m-k~.
              \end{equation}
          \end{enumerate}      
      \end{proof}

     Finally, we give an example in which the data $\{T_{\alpha}\}_{\alpha\in \Deck(\hat{N}/N)}$ contain a nonconstant map.
     \begin{example}
         Let $M:=S^{3}\times S^{1}/ \left \langle \alpha' \right \rangle$ and $N:= S^{3}/\left \langle \alpha \right \rangle$, 
     where \begin{equation}
         \alpha':S^{3}\times S^{1}\to S^{3}\times S^{1},\alpha'(x,y,e^{i\theta})=(-x,-y,e^{i(\theta+T_{\alpha}(x,y))}),
     \end{equation}
   \begin{equation}
         \alpha:S^{3}\to S^{3}, \alpha(x,y)=(-x,-y)
     \end{equation} and 
     \begin{equation}
         T_{\alpha}(x,y)=\text{Re}(x),
     \end{equation} with $(x,y)\in S^{3}\subseteq \mathbb{C}^{2},\theta\in [0,2\pi)$, and where $\text{Re}(x)$ denotes the real part of the complex number $x$.
     
     We only need to verify the cocycle condition (\ref{eqn 5.24}):
     \begin{equation}
     \begin{aligned}
         \text{Pr}_{S^{1}}\circ T_{\id}&=[0];\\
         T_{\alpha}(\alpha(x,y))+T_{\alpha}(x,y)&=T_{\alpha}(-x,-y)+T_{\alpha}(x,y)=\text{Re}(-x)+\text{Re}(x)=0.
     \end{aligned}
     \end{equation}
      By Theorem \ref{Thm 5.7}, $M$ is a principal $S^{1}$-bundle over $N$ satisfying condition (\ref{eq:difference}).
     \end{example}

  \section{Further Discussions}\label{sec: Fur}
  In this section, we outline several related results and open questions concerning Betti number bounds and topological splitting in broader geometric and topological settings.
  \subsection{Geometric Generalizations}
 We present supplementary results on the upper bounds of the first Betti number in various geometric contexts.
  \begin{theorem}[Asymptotic Cones Case, see {\cite[Theorem A]{PanYe24}}]
  Given an open manifold $(M^{m},g)$ with nonnegative Ricci curvature, suppose that an asymptotic cone of $M$ is properly contained in the Euclidean space $\mathbb{R}^{k-1}$ (where $k\ge1$). Then:
  	\begin{enumerate} [(1)]
  		\item The first Betti number satisfies $\text{b}_{1}(M)\le m-k$;
  		\item If equality holds, then $M$ is flat and isometric to either 
        \begin{itemize}
        \item $\mathbb{R}^{k}\times T^{m-k}$, or 
        \item $\mathbb{R}^{k-1}\times N^{m-k+1}$, where $N$ is also flat and diffeomorphic to the product of an open M$\ddot{o}$bius band or $T^{m-k-1}$.
        \end{itemize}
  	\end{enumerate}
  \end{theorem}

\begin{theorem}[General Ricci Limit Spaces, see {\cite[Theorem 1]{Za22}}]
	Assume that there is a sequence of closed manifolds $(M_{i}^{m},g_{i})$ with $\Ric(g_{i})\ge-(m-1)$, $\diam(M_{i}) \le D$ and $\text{b}_{1}(M_{i})\ge r$. Suppose that this sequence converges to $(X,d)$ in the Gromov–Hausdorff sense. If $(X,d)$ has a k-regular point, then:
	\begin{equation}
		r-\text{b}_{1}(X)\le m-k.
	\end{equation}
\end{theorem}
  
     We recall some known estimates for higher Betti numbers; see \cite[Theorem 9.4.6]{Petersen}.
  
          \begin{theorem}[Upper Bounded Estimation of Higher-Order Betti Numbers]
  	    Given a positive integer $m$ and nonnegative constants $k$ and $D$, there exists a constant $C>0$, depending only on $m$ and $kD^{2}$, such that for any closed manifold $(M^{m},g)$ with $\mathcal{R}m\ge-k\cdot g\odot g$ and $\diam(M)\le D$, the $i$-th Betti number satisfies:
  	    \begin{equation}
  	    	\text{b}_{i}(M)\le \binom{m}{i}\exp(C(n,kD^{2})(kD^{2})^{\frac{1}{2}}).
  	    \end{equation}
          \end{theorem}
        When $k=0$, the classification theorem for compact manifolds with nonnegative curvature operator, established by Meyer and Gallot~\cite[Theorem 10.3.7]{Petersen}, provides the equality condition.    
        \begin{theorem}
        	 For any compact manifold $ (M^{m},g) $ with a nonnegative curvature operator ($\mathcal{R}m\ge0$), if the $i$-th Betti number satisfies $\text{b}_{i}(M)=\binom{m}{i}$ for some $ i\in \{1,2,\dots,m-1\} $, then $M$ is diffeomorphic to a torus.
        \end{theorem}
    
        We propose the following conjecture.
          \begin{conjecture}
        	Given a positive integer $m$, there exists an $\epsilon(m)>0$ such that for any compact manifold $ (M^{m},g) $ with diameter $\diam(M)=1$, if the curvature operator $\mathcal{R}m\ge -\epsilon g\odot g $, then the i-th Betti number satisfies $\text{b}_{i}(M)\le \binom{m}{i} $. Moreover, the equality holds if and only if $M$ is diffeomorphic to a torus.
          \end{conjecture}
 
    \subsection{Topological Generalizations}
         A natural question arises: under what conditions does a general principal $G$-bundle split topologically or smoothly?
         \begin{conjecture}
         	Let $B$ be a topological (resp. smooth) principal
            $G$-bundle over a topological (resp. smooth) manifold $X$, where $G$ is a Lie group acting freely (resp. smoothly and freely) on $B$, and the fibration is continuous (resp. smooth). Then $B$ is homeomorphic (resp. diffeomorphic) to $X \times G$ if and only if the following condition holds: the boundary maps $ \partial_{i}: \pi_{i+1}(X)\to \pi_{i}(G)$ are trivial for all $i\ge1$, and the long homotopy exact sequence of the bundle induces short splitting exact sequences of homotopy groups:
        \begin{equation}
         \xymatrix{
         0 \ar[r] & \pi_{i}(G) \ar[r]^{l_{*}} & \pi_{i}(B) \ar[r]^{p_{*}} & \pi_{i}(X) \ar[r] & 0 
         }.
         \end{equation}
         \end{conjecture}
         
\appendix
\section{Bijection Between Smooth and Topological Principal Bundles}
In this appendix, we review the theorem on the one-to-one correspondence between smooth principal bundles and topological principal bundles with torus structure group (here, a smooth principal bundle means that the total space, fibration map, structure group, and local trivializations are all smooth), along with its proof, where the method of proof employs the Poincar$\acute{e}$ lemma in $\check{C}$ech cohomology. We include an elementary proof here for completeness and we do not claim any originality for this part. 

\begin{theorem}[See {\cite[Section 4.3 in Preliminaries]{Raghunathan}}]\label{Thm: A.1}
     Let $N^{k}$ be a closed $k$-dimensional smooth manifold and $T^{m-k}$ an $(m-k)$-dimensional torus. Denote by $\Prin^{\text{Smooth}}_{T^{m-k}}(N)$ the set of smooth equivalence classes of smooth principal $T^{m-k}$-bundles over $N$, and denote by $\Prin^{\text{Top}}_{T^{m-k}}(N)$ the set of topological equivalence classes of topological principal $T^{m-k}$-bundles over $N$. Then the natural map 
    \begin{equation}
    \begin{aligned}
        \Phi : \Prin^{\text{Smooth}}_{T^{m-k}}(N)&\to \Prin^{\text{Top}}_{T^{m-k}}(N),\\
        M^{\text{Smooth}}&\mapsto \text{Topological bundle structure on }M^{\text{Smooth}}
    \end{aligned}
    \end{equation}
    is a bijection.
\end{theorem} 

    We adopt the following notations, along with a brief explanation of the rationale behind some definitions.
\begin{enumerate}
    \item [$\bullet$] $\text{pr}_{T^{m-k}}:\mathbb{R}^{m-k}\to T^{m-k}$ denotes the standard universal covering map of $T^{m-k}$;
    
    The element $[0]$ denotes the identity of $T^{m-k}$;
    
    The metric $g _{\text{Eucl}}$ denotes the standard flat product metric on $T^{m-k}=\displaystyle\prod_{1}^{m-k}S^{1}$;

    $d$ denotes the distance function induced by $g _{\text{Eucl}}$;
    
    $\inj(T^{m-k},g _{\text{Eucl}})$ denotes the injectivity radius of $g _{\text{Eucl}}$;
    \item [$\bullet$] Choose a small positive number $\epsilon<\frac{\inj(T^{m-k},g _{\text{Eucl}})}{10000}$ such that the restriction 
    \begin{equation}
        \text{pr}_{T^{m-k}}|_{B(0, 100\epsilon)}:B(0, 100\epsilon)\to B([0], 100\epsilon)
    \end{equation}admits a smooth inverse map $\log_{T^{m-k}}$ near $[0]\in T^{m-k}$ and $\log_{T^{m-k}}([0])=0$;
    \item [$\bullet$] By compactness of $N$, we can choose a finite open cover $\{U_{i}\}_{i\in I}$ such that each $U_{i}$ is diffeomorphic to $\mathbb{R}^{k}$. We choose a partition of unity $\{\rho_{i}\}_{i\in I}$ subordinate to $\{U_{i}\}_{i\in I}$
         such that: 
         \begin{equation}
             \supp(\rho_{i})\subset U_{i},~~\text{and} \displaystyle \sum_{i\in I}\rho_{i}=1;
         \end{equation}
    \item [$\bullet$] For any continuous (resp. smooth) map $A:U_{i}\to\mathbb{R}^{m-k}$, we obtain a well-defined continuous (resp. smooth) map
    \begin{equation}
        \rho_{i}\cdot A: N\to \mathbb{R}^{m-k}
    \end{equation} on $N$;
    \item [$\bullet$] Let $\{g_{ij}\}_{i,j\in I}$ with $g_{ij}:U_{i}\cap U_{j}\to G$ be the family of continuous (resp. smooth) transition maps corresponding to the topological  (resp. smooth) principal bundle $M$ over $N$. They satisfy $g_{ii}=[0]$ and the cocycle condition:
    \begin{equation}
        g_{ij}+g_{jl}+g_{li}=[0]
    \end{equation}on $U_{i}\cap U_{j}\cap U_{l}$;
    \item [$\bullet$] $M_{1}\cong_{\text{Top bundle iso}} M_{2}$ denotes the topological equivalence of topological bundle $M_{1},M_{2}$ over $N$. 
    
    $M_{1}\cong_{\text{Diffeo bundle iso}} M_{2}$ denotes the smooth equivalence of smooth bundle $M_{1},M_{2}$ over $N$.
\end{enumerate}

We will need the following three lemmas.
\begin{lemma}[Corollary of Whitney’s Approximation Theorem {\cite[Section 6.7]{Steenrod}}]\label{lem: A.2}
    Let $U$ be a manifold. For any continuous map $A:U\to T^{m-k}$, there exists a smooth map $A':U\to T^{m-k}$ such that $d(A(x),A'(x))<\epsilon$ for all $x\in U$.
\end{lemma}

\begin{lemma}[See {\cite[Section 2, Lemma 2.10]{Steenrod}}]\label{lem: A.3}
    Let $G$ be a compact Lie group and $N$ a compact manifold. For two topological (resp. smooth) principal $G$-bundles $M_{1}$ and $M_{2}$ over $N$, let $\{g^{(a)}_{ij}\}_{i,j\in I}$($a=1,2$, $g^{(a)}_{ij}:U_{i}\cap U_{j}\to G$) be the family of continuous (resp. smooth) transition maps corresponding to $M_{a}$.
    Then $M_{1}$ and $M_{2}$ are topological (resp. smooth) bundle isomorphic if and only if there exists a family of continuous (resp. smooth) maps $\{\lambda_{i}\}_{i\in I}$ ($\lambda_{i}:U_{i}\to G$) such that 
    \begin{equation}
        g^{(1)}_{ij}=(\lambda_{i})^{-1}\cdot g^{(2)}_{ij}\cdot \lambda_{j}.
    \end{equation}
\end{lemma}
\begin{lemma}\label{lem: A.4}
    Let $M_{1}$ and $M_{2}$ be two topological principal $T^{m-k}$-bundles over $N$ and $\{g^{(a)}_{ij}\}_{i,j\in I}$($a=1,2$, $g^{(a)}_{ij}:U_{i}\cap U_{j}\to G$) the family of transition maps corresponding to $M_{a}$.
    \begin{enumerate}[(1)]
        \item If $d(g^{(1)}_{ij},g^{(2)}_{ij})<4\epsilon$ for all $i,j\in I$, then $M_{1}\cong_{\text{Top bundle iso}} M_{2}$;
        \item If $g^{(a)}_{ij}$($i,j\in I$, $a=1,2$) are all smooth and $M_{1}\cong_{\text{Top bundle iso}} M_{2}$, then $M_{1}\cong_{\text{Diffeo bundle iso}} M_{2}$.
    \end{enumerate}
\end{lemma}
\begin{proof}
    \begin{enumerate}[(1)]
        \item  Denote 
             \begin{equation}
                 \phi_{ij}:=g^{(1)}_{ij}-g^{(2)}_{ij}.
             \end{equation}
            Then $\{\phi_{ij}\}_{i,j\in I}$ satisfies the cocycle condition:
            \begin{equation}
                \begin{aligned}
                    \phi_{ij}+\phi_{jl}+\phi_{li}&=(g^{(1)}_{ij}-g^{(2)}_{ij})+(g^{(1)}_{jl}-g^{(2)}_{jl})+(g^{(1)}_{li}-g^{(2)}_{li})\\
                    &=(g^{(1)}_{ij}+g^{(1)}_{jl}+g^{(1)}_{li})-(g^{(2)}_{ij}+g^{(2)}_{jl}+g^{(2)}_{li})=[0],
                \end{aligned}
            \end{equation}
              and
              \begin{equation}
                d(\phi_{ij},[0])=d(g^{(1)}_{ij}-g^{(2)}_{ij},[0])=d(g^{(1)}_{ij},g^{(2))}_{ij})<4\epsilon.   
              \end{equation}

              We need to find a family of maps $\{\mu_{i}\}_{i\in I}$($\mu_{i}:U_{i}\to T^{m-k}$) such that 
           \begin{equation}
               \phi_{ij}=\mu_{j}-\mu_{i}.
           \end{equation}

            Define 
            \begin{equation}
                \mu_{i}(x):=\text{pr}_{T^{m-k}}(-\displaystyle\sum_{j\in I}\rho_{j}(x)\cdot(\log_{T^{m-k}}\circ\phi_{ij})(x))
            \end{equation} for all $x\in U_{i}$.

           Then $\{\mu_{i}\}_{i\in I}$ is a family of maps satisfies:
           \begin{equation}
           \begin{aligned}
               \mu_{j}-\mu_{i}&=\text{pr}_{T^{m-k}}(-\displaystyle\sum_{l\in I}\rho_{l}\cdot(\log_{T^{m-k}}\circ\phi_{jl}))-\text{pr}_{T^{m-k}}(-\displaystyle\sum_{l\in I}\rho_{l}\cdot(\log_{T^{m-k}}\circ\phi_{il}))\\
               &=\text{pr}_{T^{m-k}}(\displaystyle\sum_{l\in I}\rho_{l}\cdot(\log_{T^{m-k}}(\phi_{il}-\phi_{jl})))\\
               &=\text{pr}_{T^{m-k}}(\displaystyle\sum_{l\in I}\rho_{l}\cdot(\log_{T^{m-k}}\circ\phi_{ij}))\\
               &=\text{pr}_{T^{m-k}}\circ\log_{T^{m-k}}\circ\phi_{ij}=\phi_{ij}.
           \end{aligned}   
           \end{equation}

           Consequently, the family $\{\mu_{i}\}_{i\in I}$ of maps satisfies 
           \begin{equation}
             g^{(1)}_{ij}-g^{(2)}_{ij}=\mu_{j}-\mu_{i}
          \end{equation} on $U_{i}\cap U_{j}$.
               By Lemma \ref{lem: A.3}, we obtain $M_{1}\cong_{\text{Top bundle iso}} M_{2}$.
        \bigskip
        \bigskip
          \item By Lemma \ref{lem: A.3}, there exists a family of continuous maps 
          \begin{equation}
              \lambda^{\text{Top}}_{i}:U_{i}\to T^{m-k}
          \end{equation}
          such that 
          \begin{equation}
             g^{(1)}_{ij}-g^{(2)}_{ij}=\lambda^{\text{Top}}_{j}-\lambda^{\text{Top}}_{i}
          \end{equation} on $U_{i}\cap U_{j}$.
          
          By Lemma \ref{lem: A.2}, there exists a family of smooth maps 
          \begin{equation}
              \lambda^{\text{Smooth}}_{i}: U_{i}\to T^{m-k}
          \end{equation}
          such that $d(\lambda^{\text{Top}}_{i},\lambda^{\text{Smooth}}_{i})<\epsilon$.
           Define \begin{equation}
               \tau_{ij}:=g^{(1)}_{ij}-g^{(2)}_{ij}-\lambda^{\text{Smooth}}_{j}+\lambda^{\text{Smooth}}_{i}.
           \end{equation}
           Then $\{\tau_{ij}\}_{i,j\in I}$ satisfies the cocycle condition:
           \begin{equation}
               \begin{aligned}
                   &\tau_{ij}+\tau_{jl}+\tau_{li}\\
                   =&(g^{(1)}_{ij}-g^{(2)}_{ij}-\lambda^{\text{Smooth}}_{j}+\lambda^{\text{Smooth}}_{i})+(g^{(1)}_{jl}-g^{(2)}_{jl}-\lambda^{\text{Smooth}}_{l}+\lambda^{\text{Smooth}}_{j})
                   +(g^{(1)}_{li}-g^{(2)}_{li}-\lambda^{\text{Smooth}}_{i}+\lambda^{\text{Smooth}}_{l})\\
                   =&(g^{(1)}_{ij}+g^{(1)}_{jl}+g^{(1)}_{li})-(g^{(2)}_{ij}+g^{(2)}_{jl}+g^{(2)}_{li})=[0],
               \end{aligned}
           \end{equation}
           and
           \begin{equation}
           \begin{aligned}
                 d(\tau_{ij},[0])=&d(g^{(1)}_{ij}-g^{(2)}_{ij}-\lambda^{\text{Smooth}}_{j}+\lambda^{\text{Smooth}}_{i},g^{(1)}_{ij}-g^{(2)}_{ij}-\lambda^{\text{Top}}_{j}+\lambda^{\text{Top}}_{i})\\
                 =&d(\lambda^{\text{Smooth}}_{i}-\lambda^{\text{Smooth}}_{j},\lambda^{\text{Top}}_{i}-\lambda^{\text{Top}}_{j})\\
                 \le& d(\lambda^{\text{Smooth}}_{i}-\lambda^{\text{Smooth}}_{j},\lambda^{\text{Smooth}}_{i}-\lambda^{\text{Top}}_{j})+d(\lambda^{\text{Smooth}}_{i}-\lambda^{\text{Top}}_{j},\lambda^{\text{Top}}_{i}-\lambda^{\text{Top}}_{j})\\
                 =&d(\lambda^{\text{Smooth}}_{j},\lambda^{\text{Top}}_{j})+d(\lambda^{\text{Smooth}}_{i},\lambda^{\text{Top}}_{i})<2\epsilon.
           \end{aligned}
           \end{equation}

           We need to find a family of smooth maps $\{\lambda^{'}_{i}\}_{i\in I}$($\lambda^{'}_{i}:U_{i}\to T^{m-k}$) such that 
           \begin{equation}
               \tau_{ij}=\lambda^{'}_{j}-\lambda^{'}_{i}.
           \end{equation}
           Define 
           \begin{equation}
               \lambda^{'}_{i}(x):=\text{pr}_{T^{m-k}}(-\displaystyle\sum_{j\in I}\rho_{j}(x)\cdot(\log_{T^{m-k}}\circ\tau_{ij})(x))
           \end{equation}for all $x\in U_{i}$.
           Then $\{\lambda^{'}_{i}\}_{i\in I}$ is a family of smooth maps satisfies:
           \begin{equation}
           \begin{aligned}
               \lambda^{'}_{j}-\lambda^{'}_{i}&=\text{pr}_{T^{m-k}}(-\displaystyle\sum_{l\in I}\rho_{l}\cdot(\log_{T^{m-k}}\circ\tau_{jl}))-\text{pr}_{T^{m-k}}(-\displaystyle\sum_{l\in I}\rho_{l}\cdot(\log_{T^{m-k}}\circ\tau_{il}))\\
               &=\text{pr}_{T^{m-k}}(\displaystyle\sum_{l\in I}\rho_{l}\cdot(\log_{T^{m-k}}(\tau_{il}-\tau_{jl})))\\
               &=\text{pr}_{T^{m-k}}(\displaystyle\sum_{l\in I}\rho_{l}\cdot(\log_{T^{m-k}}\circ\tau_{ij}))\\
               &=\text{pr}_{T^{m-k}}\circ\log_{T^{m-k}}\circ\tau_{ij}=\tau_{ij}.
           \end{aligned}   
           \end{equation}
           Consequently, the family $\{\lambda^{\text{Smooth}}_{i}+\lambda^{'}_{i}\}_{i\in I}$ of smooth maps satisfies 
           \begin{equation}
             g^{(1)}_{ij}-g^{(2)}_{ij}=(\lambda^{\text{Smooth}}_{j}+\lambda^{'}_{j})-(\lambda^{\text{Smooth}}_{i}+\lambda^{'}_{i})
          \end{equation} on $U_{i}\cap U_{j}$. By Lemma \ref{lem: A.3}, we obtain $M_{1}\cong_{\text{Diffeo bundle iso}} M_{2}$.
    \end{enumerate}
\end{proof}

\begin{proof of theorem A.1} 
        Injectivity of $\Phi$ follows immediately from Lemma \ref{lem: A.4} (2).

        To prove surjectivity of $\Phi$, let $M^{\text{Top}}\in \Prin^{\text{Top}}_{T^{m-k}}(N)$ be an arbitrary topological principal $T^{m-k}$-bundle. We need to find a smooth principal $T^{m-k}$-bundle $M^{\text{Smooth}}$ that is topologically  equivalent to $M^{\text{Top}}$.

        Let $g^{\text{Top}}_{ij}:U_{i}\cap U_{j}\to T^{m-k}$ be the continuous transition maps of the principal $T^{m-k}$-bundle $M^{\text{Top}}$. They satisfy $g^{\text{Top}}_{ii}=[0]$ and the cocycle condition
         \begin{equation}
             g^{\text{Top}}_{ij}+g^{\text{Top}}_{jl}+g^{\text{Top}}_{li}=[0]
         \end{equation}on $U_{i}\cap U_{j}\cap U_{l}$.
         
         By Lemma \ref{lem: A.2}, there exist smooth maps 
         \begin{equation}
             g^{\text{Smooth}}_{ij}:U_{i}\cap U_{j}\to T^{m-k}, i,j\in I
         \end{equation} such that $g^{\text{Smooth}}_{ij}+g^{\text{Smooth}}_{ji}=[0]$ and $d(g^{\text{Smooth}}_{ij}(x),g^{\text{Top}}_{ij}(x))<\epsilon$ for all $x\in U_{i}\cap U_{j}$. 
          Denote \begin{equation}
            \eta_{ijl}:=g^{\text{Smooth}}_{ij}+g^{\text{Smooth}}_{jl}+g^{\text{Smooth}}_{li}
        \end{equation}on $U_{i}\cap U_{j}\cap U_{k}$.
       Then $\{\eta_{ijl}\}_{i,j,l\in I}$ is a family of smooth maps satisfies 
       \begin{equation}
             \begin{aligned}
                 \eta_{ijs}-\eta_{ijl}+\eta_{isl}-\eta_{jsl}
                 =&(g^{\text{Smooth}}_{ij}+g^{\text{Smooth}}_{js}+g^{\text{Smooth}}_{si})-(g^{\text{Smooth}}_{ij}+g^{\text{Smooth}}_{jl}+g^{\text{Smooth}}_{li})\\
                 +&(g^{\text{Smooth}}_{is}+g^{\text{Smooth}}_{sl}+g^{\text{Smooth}}_{li})-(g^{\text{Smooth}}_{js}+g^{\text{Smooth}}_{sl}+g^{\text{Smooth}}_{lj})\\
                 =&g^{\text{Smooth}}_{si}-g^{\text{Smooth}}_{jl}+g^{\text{Smooth}}_{is}-g^{\text{Smooth}}_{lj}=[0]
             \end{aligned}
         \end{equation}on $U_{i}\cap U_{j}\cap U_{s}\cap U_{l}$,
       and by the bi-invariance of $g _{\text{Eucl}}$, we have:
         \begin{equation}
             \begin{aligned}
                 d(\eta_{ijl},[0])=&d(g^{\text{Smooth}}_{ij}+ g^{\text{Smooth}}_{jl}+ g^{\text{Smooth}}_{li},g^{\text{Top}}_{ij}+ g^{\text{Top}}_{jl}+ g^{\text{Top}}_{li})\\
                 \le &d(g^{\text{Smooth}}_{ij}+ g^{\text{Smooth}}_{jl}+ g^{\text{Smooth}}_{li},g^{\text{Smooth}}_{ij}+ g^{\text{Smooth}}_{jl}+ g^{\text{Top}}_{li})\\
                 +&d(g^{\text{Smooth}}_{ij}+ g^{\text{Smooth}}_{jl}+ g^{\text{Top}}_{li},g^{\text{Smooth}}_{ij}+ g^{\text{Top}}_{jl}+ g^{\text{Top}}_{li})\\
                 +&d(g^{\text{Smooth}}_{ij}+ g^{\text{Top}}_{jl}+ g^{\text{Top}}_{li},g^{\text{Top}}_{ij}+ g^{\text{Top}}_{jl}+ g^{\text{Top}}_{li})\\
=&d(g^{\text{Smooth}}_{ij},g^{\text{Top}}_{ij})+d(g^{\text{Smooth}}_{jl},g^{\text{Top}}_{jl})+d(g^{\text{Smooth}}_{li},g^{\text{Top}}_{li})\\
                 <&3\epsilon.
             \end{aligned}
         \end{equation}

         We need to find a family of smooth maps $\{\lambda_{ij}\}_{i,j\in I}$ ($\lambda_{ij}:U_{ij}\to T^{m-k}$) such that 
         \begin{equation}\label{cocycle}
         g^{\text{Smooth}}_{ij}+\lambda_{ij}+g^{\text{Smooth}}_{jl}+
         \lambda_{jl}+g^{\text{Smooth}}_{li}+\lambda_{li}=[0]
         \end{equation} 
         for all $i,j,l\in I$, and $d(\lambda_{ij},[0])<3\epsilon$.
         
        Define \begin{equation}
            \theta_{ij}(x):=\displaystyle\sum_{l\in I}\rho_{l}(x)\cdot(\log_{T^{m-k}}\circ\eta_{ijl})(x)
        \end{equation}for all $x\in U_{i}\cap U_{j}$.

        We compute 
        \begin{equation}
        \begin{aligned}
             \theta_{ij}+\theta_{jl}+\theta_{li}=&\displaystyle\sum_{s\in I}[\rho_{s}\cdot(\log_{T^{m-k}}\circ\eta_{ijs})+\rho_{s}\cdot(\log_{T^{m-k}}\circ\eta_{jls})+\rho_{s}\cdot(\log_{T^{m-k}}\circ\eta_{lis})]\\
             =&\displaystyle\sum_{s\in I}\rho_{s}\cdot[(\log_{T^{m-k}}\circ\eta_{ijs})+(\log_{T^{m-k}}\circ\eta_{jls})+(\log_{T^{m-k}}\circ\eta_{lis})]\\
             =&\displaystyle\sum_{s\in I}\rho_{s}\cdot(\log_{T^{m-k}}(\eta_{ijs}+\eta_{jls}+\theta_{lis}))\\
             =&\displaystyle\sum_{s\in I}\rho_{s}\cdot(\log_{T^{m-k}}(\eta_{ijl}))=\log_{T^{m-k}}(\eta_{ijl}).
        \end{aligned}
        \end{equation}
        
        Let $\lambda_{ij}:=\text{pr}_{T^{m-k}}(-\theta_{ij})$, then
        \begin{equation}
        \begin{aligned}
              g^{\text{Smooth}}_{ij}+\lambda_{ij}+g^{\text{Smooth}}_{jl}+
         \lambda_{jl}+g^{\text{Smooth}}_{li}+\lambda_{li}=&\eta_{ijl}+\text{pr}_{T^{m-k}}(-\theta_{ij})+\text{pr}_{T^{m-k}}(-\theta_{jl})+\text{pr}_{T^{m-k}}(-\theta_{li})\\
         =&\text{pr}_{T^{m-k}}(\log_{T^{m-k}}(\eta_{ijl})-(\theta_{ij}+\theta_{jl}+\theta_{li}))\\
         =&\text{pr}_{T^{m-k}}(\log_{T^{m-k}}(\eta_{ijl})-\log_{T^{m-k}}(\eta_{ijl}))=[0],
        \end{aligned} 
        \end{equation}
         and
        \begin{equation}
         d(\lambda_{ij},[0])=|\theta_{ij}|\le \displaystyle\sum_{l\in I}|\rho_{l}|\cdot|\log_{T^{m-k}}\circ\eta_{ijl}|=\displaystyle\sum_{l\in I}\rho_{l}\cdot d(\eta_{ijl},[0])<\displaystyle\sum_{l\in I}\rho_{l}\cdot 3\epsilon=3\epsilon.
        \end{equation}
          Consequently, (\ref{cocycle}) holds, and we obtain a family of 
          transition maps $\{g^{\text{Smooth}}_{ij}+\lambda_{ij}\}_{i,j\in I}$ that continuously deforms to $\{g^{\text{Top}}_{ij}\}_{i,j\in I}$. By taking the quotient via these smoothed transition maps $\{g^{\text{Smooth}}_{ij}+\lambda_{ij}\}_{i,j\in I}$, we obtain a smooth principal bundle, which we denote by $M^{\text{Smooth}}$.
          
          We estimate 
          \begin{equation}
          \begin{aligned}
              d(g^{\text{Top}}_{ij},g^{\text{Smooth}}_{ij}+\lambda_{ij})&\le d(g^{\text{Top}}_{ij},g^{\text{Smooth}}_{ij})+d(g^{\text{Smooth}}_{ij},g^{\text{Smooth}}_{ij}+\lambda_{ij})\\
              &=d(g^{\text{Top}}_{ij},g^{\text{Smooth}}_{ij})+d([0],\lambda_{ij})<\epsilon+3\epsilon=4\epsilon.
          \end{aligned}            
          \end{equation}
             By Lemma \ref{lem: A.4} (1), $M^{\text{Smooth}}$ is topologically equivalent to $M^{\text{Top}}$ as a principal bundle.
\end{proof of theorem A.1}


\bigskip
\begin{thebibliography}{100}
	
	\bibitem{Hatcher} Allen Hatcher, \emph{Algebraic topology}, Cambridge University Press. (2002).
	
	\bibitem{Bochner} Salomon Bochner, \emph{Vector fields and Ricci curvature}, Bull. Amer. Math. Soc. 52 (1946), 776-797.

    \bibitem{BY} Salomon Bochner and Kentaro Yano, \emph{Curvature and Betti numbers}, Annals of Mathematics Studies, No. 32. Princeton University Press, Princeton, N.J. (1953).

    \bibitem{BS} Marianna C. Bonanome, Margaret H. Dean, Judith Putnam Dean, \emph{A sampling of remarkable groups : Thompson's, self-similar, lamplighter, and Baumslag-Solitar}, Cham, Switzerland : Birkhäuser, (2018).

    \bibitem{CGI} Jeff Cheeger and Mikhail Gromov, \emph{Collapsing Riemannian manifolds while keeping their curvature bounded}, I. J. Differential Geom. 23 (1986), no. 3, 309-346.

    \bibitem{CGII} Jeff Cheeger and Mikhail Gromov, \emph{Collapsing Riemannian manifolds while keeping their curvature bounded}, II. J. Differential Geom. 32 (1990), no. 1, 269-298.

    \bibitem{CFG92} Jeff Cheeger, Kenji Fukaya and Mikhail Gromov, \emph{Nilpotent structures and invariant metrics on collapsed manifolds}, J. Amer. Math. Soc. 5 (1992), no. 2, 327-372.

    \bibitem{Fukaya87ld} Kenji Fukaya, \emph{Collapsing Riemannian manifolds to ones of lower dimensions}, J. Differential Geom. 25 (1987), no. 1, 139-156.
 
    \bibitem{Fukaya88} Kenji Fukaya, \emph{A boundary of the set of the Riemannian manifolds with bounded curvatures and diameters}, J. Differential Geom. 28 (1988), no. 1, 1-21.

    \bibitem{Fukaya89} Kenji Fukaya, \emph{Collapsing Riemannian manifolds to ones with lower dimension II}, J. Math. Soc. Japan, 41 (1989), 333-356.

    \bibitem{FY92} Kenji Fukaya and Takao Yamaguchi, \emph{The fundamental groups of almost non-negatively curved manifolds}, Ann. of Math. (2) 136 (1992), no. 2, 253-333.
    
    \bibitem{GompfStipsicz} R.~Gompf and A.~Stipsicz, \emph{4-Manifolds and Kirby Calculus}, Graduate Studies in Mathematics, Vol. 20, American Mathematical Society, Providence, RI, (1999).

    \bibitem{Gromov78b} Mikhail Gromov, \emph{Almost flat manifolds}, J. Differential Geom. 13 (1978), no. 2, 231-241. 

    \bibitem{GLP} Mikhail Gromov, \emph{Structures m\'etriques pour les vari\'et\'es Riemanniennes}, Editions Cedic, Paris. (1981).  

	\bibitem{HuangWang20.1} Shaosai Huang and Bing Wang, \emph{Rigidity of the first Betti number via Ricci flow smoothing}, Sci. China Math. (2025). 

    \bibitem{HuangWang20.2} Shaosai Huang and Bing Wang, \emph{Ricci flow smoothing for locally collapsing manifolds}, arXiv:2008.09956.
    
	\bibitem{Huang20} Hongzhi Huang, \emph{Fibrations, and stability for compact group actions on manifolds with local bounded Ricci covering geometry}, arXiv:2002.07383.
	
	\bibitem{HKRX18} Hongzhi Huang; Lingling Kong; Xiaochun Rong; Shicheng Xu, \emph{Collapsed Manifolds With Ricci Bounded Covering Geometry}, arXiv:1808.03774.

    \bibitem{MS} John W. Milnor and James D. Stasheff, \emph{Characteristic Classes}, Princeton University Press. (1974).
    
    \bibitem{KN63} Shoshichi Kobayashi and Katsumi Nomizu, \emph{Foundations of Differential Geometry, Volume I}, New York: Interscience Publishers. (1963). 
    
    \bibitem{KW11} Vitali Kapovitch and Burkhard Wilking, \emph{Structure of fundamental groups of manifolds with Ricci curvature bounded below}, arXiv:1105.5955.
    
    \bibitem{NaberZhang} Aaron Naber and Ruobing Zhang, \emph{Topology and $\varepsilon$-regularity theorems on collapsed manifolds with Ricci curvature bounds}, Geometry and Topology 20 (2016), no. 5, 2575-2664.

    \bibitem{Petersen} Peter Petersen, \emph{Riemannian Geometry} 3rd edition, New York: Springer-Verlag, GTM 171. (2016).
    
    \bibitem{PanYe24} Jiayin Pan and Zhu Ye, \emph{Nonnegative Ricci curvature, splitting at infinity, and first Betti number rigidity}, arXiv:2404.10145. 

    \bibitem{Raghunathan} Madabusi S. Raghunathan, \emph{Discrete subgroups of Lie groups}, Springer-Verlag. (1972).
    
	\bibitem{Rong03} Xiaochun Rong, \emph{Collapsed Riemannian manifolds with bounded sectional curvature}, arXiv:math/0304267.
	
	\bibitem{Rong19} Xiaochun Rong, \emph{A New Proof of the Gromov's Theorem on Almost Flat Manifolds}, arXiv:1906.03377.

    \bibitem{Rong22} Xiaochun Rong, \emph{Collapsed manifolds with local Ricci bounded covering geometry}, arXiv:2211.09998. 
    
    \bibitem{Ruh} Ernst A. Ruh, \emph{Almost flat manifolds}, J. Differential Geom. 17 (1982), no. 1, 1-14.

    \bibitem{Steenrod} N. Steenrod, \emph{The Topology of Fibre Bundles}, Princeton University Press, New Jersey. (1951).
	
    \bibitem{Za22} Sergio Zamora, \emph{First Betti Number and Collapse}, arXiv:2209.12628. 
\end{thebibliography}
\end{document}